\documentclass{article}
\usepackage{graphicx}
\graphicspath{Fig/}

\usepackage[utf8]{inputenc}
\usepackage[margin=1in]{geometry}
\usepackage{enumitem}
\usepackage{authblk}

\usepackage{tikz}
\usetikzlibrary{patterns}
\usetikzlibrary{backgrounds,scopes}

\usepackage{hyperref}
\usepackage{hypcap}

\usepackage[ruled,vlined]{algorithm2e}
\usepackage{algorithmic}

\usepackage{float}
\usepackage{caption}
\usepackage{subcaption}
\usepackage{setspace}
\usepackage{tcolorbox}

\usepackage{amsmath,amsfonts,amssymb,mathtools, amsthm}
\usepackage[english]{babel}
\newtheorem{prop}{Proposition}
\newtheorem{theorem}{Theorem}

\newtheorem{coro}{Corollary}
\newtheorem{lemma}{Lemma}
\newtheorem*{remark}{Remark}
\numberwithin{equation}{section}

\title{Weak solutions for weak turbulence models in electrostatic plasmas}

\author[*]{Kun Huang}
\author[$\dag$]{Irene M. Gamba}

\affil[*]{Institute for Fusion Studies, University of Texas at Austin}
\affil[$\dag$]{Oden Institute for Computational Sciences and Engineering, University of Texas at Austin}

\date{}

\usepackage{subfiles} 

\begin{document}
\maketitle

\begin{abstract}
    The weak turbulence model, also known as the quasilinear theory in plasma physics, has been a cornerstone in modeling resonant particle-wave interactions in plasmas. Stemming from the Vlasov-Poisson/Maxwell system under the weak turbulence assumption, incorporating the random phase approximation and ergodicity, it takes the form of a integro-differential system, which describes the time dynamics of particle distribution and wave spectrum.
    The interaction between particles and waves can be treated as a stochastic process, whose transition probability bridges the momentum space and the spectral space. Consequently, the interaction operators resemble  the nonlinear and nonlocal collision forms, such as those of Boltzmann and Landau type. For those classical models, there have been results on well-posedness and regularity of solutions. However, to our knowledge, there is no preceding rigorous analytical work for the quasilinear theory addressed in this manuscript.

        In this paper, we establish the existence of global weak solutions to the one-dimensional electrostatic quasilinear model for nonnegative, smooth, rapidly decaying initial particle distributions and nonnegative, smooth wave spectra compactly supported in a positive wave-number band. The construction uses the resonance condition to reduce the coupled system to a scalar degenerate parabolic equation, together with uniform estimates and strong convergence of the gradients of regularized approximations. When the initial wave spectrum is positive throughout a single band, we further characterize the equilibrium selected by this approximation procedure. The equilibrium uniquely minimizes a quadratic functional over a nonnegative cone and admits a geometric construction through the least concave majorant of a primitive determined by the initial data. The limiting particle distribution is non-increasing on the corresponding resonant interval and forms plateaus wherever the equilibrium auxiliary function is positive. A relative-entropy argument establishes long-time convergence to this equilibrium.

    \vspace{6pt}\textbf{Keywords:} kinetic equation, porous medium equation, weak turbulence model, quasilinear theory, plasma physics
\end{abstract}

\section{Introduction} 

    The study of the quasilinear theory is motivated by the limitations of the Landau damping theory \cite{landau1946vibrations}. As we know, waves are damped by particles whose distribution satisfies the Penrose criterion. However, this criterion confines the initial particle distribution in a narrow regime within the function space. Outside this regime, where waves do not simply get damped, a back-reaction of the waves on the particles becomes significant. To describe such back-reaction, in 1960s, Vedenov et al.  \cite{vedenov1961nonlinear} and Drummond et al.  \cite{drummond1962non} introduced the quasilinear theory. 

    The name ``quasilinear theory" is commonly used in plasma physics, as it accurately reveals the feature of the model equations. Nevertheless, from a mathematician’s perspective, its synonym ``weak turbulence model" may be easier to grasp. Indeed, the quasilinear theory plays the same role for the Vlasov-Poisson equations as Reynolds' averaging does for the Navier–Stokes equations, given that it describes the evolution of an ``averaged" probability density function by encapsulating all perturbations into another term. For a derivation in mathematicians’ style, we refer the readers to the recent paper by Bardos and Besse  \cite{bardos2021diffusion}, which investigates the diffusion limit of the Vlasov-Poisson system and formally derives the diffusion matrix associated with the quasilinear theory in the regime where the typical autocorrelation time of particles goes to infinity.
    
    For an interpretation in quantum mechanical language from a probabilistic point of view, the book of Thorne and Blandford  \cite{thorne2017modern} offers a very helpful framework for the readers, where this system follows from the master equations for the stochastic emission/absorption process of plasmons (quantized plasma oscillation). 

    Since proposed \cite{vedenov1961nonlinear, drummond1962non} in 1960s, the quasilinear theory has been widely used in plasma physics to describe particle-wave interactions. Despite its importance in applications, the mathematical analysis of this model has received limited attention, and its well-posedness remains an open problem. This paper is the first attempt to address this gap by proving the existence of weak solutions to this integral-differential system.

In this paper, we study the one-dimensional electrostatic model with
a prescribed dispersion relation. A useful feature of this setting
is that the resonance condition permits a reduction to a scalar
equation. Ivanov and Rudakov \cite{ivanov1967quasilinear} formally
introduced an auxiliary function from which both the particle
distribution and the wave spectrum can be recovered. This function
satisfies a degenerate diffusion equation of porous-medium type with
nonlinear source terms. The reduction provides a route to the
analysis of the coupled system, but requires sufficient regularity
and compactness to justify both the scalar equation and the
reconstruction of the original unknowns.

The first part of the paper establishes the existence of global weak
solutions for nonnegative, smooth, rapidly decaying initial particle
distributions and nonnegative, smooth initial wave spectra compactly
supported in a positive wave-number band. We begin with analyzing a series of approximate problems with parameter $n$ and try to establish regularity estimates uniform in $n$, which allows us to pass the limit to infinity. Such technique is well explained in the book of Vazquez \cite{vazquez2007porous}. A major challenge arises in obtaining the strong convergence of the gradient square term, which is necessary for passing to the limit. Similar issues have been tackled by Abdellaoui, Peral and Walias \cite{abdellaoui2015some}, where the nonlinear source contains a gradient term to a positive power. A significant difference is that they require positive source, while in our problem, the gradient square term has a minus sign. Nevertheless, their methods for almost everywhere convergence remain applicable, provided we can find an alternative way to obtain suitable estimates.

The second part addresses the equilibrium selected by the evolution.
We assume an initial wave spectrum that
is positive throughout one wave-number band. The degeneracy of the
equation allows multiple stationary states, so the stationary
equation alone does not identify the asymptotic profile. We characterize the selected equilibrium as the
unique minimizer of the quadratic energy functional. On top of that, through a geometric description, we are able to construct explicitly the shape of equilibrium state for any given initial data by taking the least concave majorant.

We then prove that every global subsequential limit of our
approximations posed on $I$ converges to this variationally selected
equilibrium. The proof uses a regularized relative-entropy identity
to compare the evolving solution with smooth admissible profiles. An energy-gap
inequality then gives strong convergence of the auxiliary function.

The paper is organized as follows. We first introduce the physical
model, derive the scalar reduction, and state the main results.
Part~I constructs global weak solutions: Section~\ref{sec_extension}
introduces the approximate problems and their extensions,
Section~\ref{sec_regularity} establishes uniform estimates,
Section~\ref{sec_convergence} proves the required compactness, and
Section~\ref{sec_existence} completes the existence argument.
Part~II studies the asymptotic behavior.
Section~\ref{sec_equilibrium} characterizes the selected equilibrium
through the variational problem and the least concave majorant,
and Section~\ref{sec_longtime} proves convergence to this equilibrium.

    \section{Background}\label{sec_bg}
The study of waves in plasmas dates back to 1946, when Landau \cite{landau1946vibrations} calculated the spectral form of Vlasov-Poisson equations. He discovered that the frequency of oscillations in a plasma contains an  imaginary part, which means their amplitude changes in time. Let $f(\mathbf{p},t): \mathbb{R}^{d}_{p}\times \mathbb{R}^+ \rightarrow [0,+\infty)$ be the particle probability density function, and let $W(\mathbf{k},t): \mathbb{R}^{d}_{k}\times \mathbb{R}^+\rightarrow [0,+\infty)$ be the wave spectral energy density, Landau's work \cite{landau1946vibrations} predicts that the waves' spectrum evolves as follows:
\begin{align}
    \label{eq_landau}
    &\partial_{t}\ W(\mathbf{k},t)=\left[\int_{\mathbb{R}_{p}^{d}}\left(\nabla_{p}f(\mathbf{p},t)\right)\cdot\left(\hat{\mathbf{k}}\otimes\hat{\mathbf{k}}\right)\delta(\omega(\mathbf{k};f)-\nabla_{p}E(\mathbf{p})\cdot\mathbf{k})\cdot\left(\nabla_{p}E(\mathbf{p})\right)d\mathbf{p}\right]W(\mathbf{k},t),\nonumber\\[3pt]
    &\textrm{with}\ W(\mathbf{k},0)=W_0(\mathbf{k}), \quad\textrm{ initial data in wave number space.}                                 \end{align}
where $E(\mathbf{p}) = \frac{\mathbf{p}^{2}}{2m_{e}}$ is the kinetic energy of a single particle, and $\omega(\mathbf{k})$ is the dispersion relation of plasma waves, which depends on the medium, i.e. the particle \textit{pdf}. It can be observed that for particular particle distribution $f$, all the waves are damped, such effect is called ``Landau damping".

If the initial particle distribution $f$ induces excitation instead of relaxation, the waves will grow exponentially, which apparently will not last long, given that the total energy is finite. In that case, we need a theory to describe the back-reaction of the waves on the particles. According to Vedenov et al.  \cite{vedenov1961nonlinear} and Drummond et al.  \cite{drummond1962non}, the interaction with waves results in a diffusion effect: 
\begin{align}\label{eq_vedenov}
    &\partial_{t}f(\mathbf{p},t)=\nabla_{p}\cdot\left(\left[\int_{\mathbb{R}_{k}^{d}}W(\mathbf{k},t)\left(\hat{\mathbf{k}}\otimes\hat{\mathbf{k}}\right)\delta(\omega(\mathbf{k};f)-\nabla_{p}E(\mathbf{p})\cdot\mathbf{k})d\mathbf{k}\right]\cdot\nabla_{p}f(\mathbf{p},t)\right)\nonumber\\[3pt]
    &\textrm{with}\ f(\mathbf{p},0)=f_0(\mathbf{p}), \quad\textrm{ initial data in momentum space.}
\end{align}

As can be observed in numerical simulations \cite{grognard1985propagation}, the diffusion causes anisotropic flattening of particle distribution $f$, i.e. a smaller gradient $\nabla_{p}f$, which is why the growth of waves will not last long.

Equation (\ref{eq_landau}) and Equation (\ref{eq_vedenov}) combined together is a closed self-consistent system:
\begin{equation} \label{QLT system}
    \left\{ \begin{aligned}\partial_{t}f(\mathbf{p},t)&=\nabla_{p}\cdot\left(\left[\int_{\mathbb{R}_{k}^{d}}W(\mathbf{k},t)\Phi(\mathbf{p},\mathbf{k})d\mathbf{k}\right]\cdot\nabla_{p}f(\mathbf{p},t)\right),\\
    \partial_{t}W(\mathbf{k},t)&=\left[\int_{\mathbb{R}_{p}^{d}}\left(\nabla_{p}f(\mathbf{p},t)\right)\cdot\Phi(\mathbf{p},\mathbf{k})\cdot\left(\nabla_{p}E(\mathbf{p})\right)d\mathbf{p}\right]W(\mathbf{k},t),
    \end{aligned}
    \right.
\end{equation}
where
\begin{equation}
    \Phi(\mathbf{p},\mathbf{k}) = \left(\hat{\mathbf{k}}\otimes\hat{\mathbf{k}}\right)\delta(\omega(\mathbf{k};f)-\nabla_{p}E(\mathbf{p})\cdot\mathbf{k}).
\end{equation}

\begin{remark}
    For magnetized plasma, we have $\Phi(\mathbf{p},\mathbf{k})=\left(\mathbf{\beta}\otimes\beta\right)\mathcal{B}(\mathbf{p},\mathbf{k})$, see  \cite{huang2023conservative}.
\end{remark}

System \eqref{QLT system} conserves total mass, total momentum and total energy. Indeed, its weak form reads
    \begin{equation*}
    \begin{split}
        &\partial_{t}\left(\int_{\mathbb{R}_{p}^{d}} f(\mathbf{p},t) \phi(\mathbf{p}) + \int_{\mathbb{R}_{k}^{d}} W(\mathbf{k},t)\eta(\mathbf{k})\right) \\
        =& \iint_{\mathbb{R}_{p}^{d}\mathbb{R}_{k}^{d}} \left(\nabla_{p}f(\mathbf{p},t)\right)\cdot W(\mathbf{k})\Phi(\mathbf{p},\mathbf{k})\cdot\left(\eta(\mathbf{k})\nabla_{p}E(\mathbf{p}) - \nabla_{p}\phi(\mathbf{p})\right)\\
        =&\iint_{\mathbb{R}_{p}^{d}\mathbb{R}_{k}^{d}} \left(\hat{\mathbf{k}}\cdot\nabla_{p}f(\mathbf{p},t)\right)\cdot W(\mathbf{k})\delta\left(\omega(\mathbf{k};f)-\nabla_{p}E(\mathbf{p})\cdot\mathbf{k}\right)\cdot\left(\eta(\mathbf{k})\hat{\mathbf{k}}\cdot\nabla_{p}E(\mathbf{p}) - \hat{\mathbf{k}}\cdot\nabla_{p}\phi(\mathbf{p})\right)
    \end{split}
    \end{equation*}
    Let $\eta = 0$ and $\phi = 1$, we obtain mass conservation
    \begin{equation*}
        \partial_{t}\left(\int_{\mathbb{R}_{p}^{d}} f(\mathbf{p},t) \right)  = 0;
    \end{equation*}
    Let $\eta = \frac{\mathbf{k}}{\omega}$ and $\phi =\mathbf{p}$, we obtain momentum conservation
    \begin{equation*}
        \partial_{t}\left(\int_{\mathbb{R}_{p}^{d}} f(\mathbf{p},t) \mathbf{p} + \int_{\mathbb{R}_{k}^{d}} W(\mathbf{k},t)\frac{\mathbf{k}}{\omega}\right)  = \mathbf{0},
    \end{equation*}
    using the relation that for any $G$,
    \begin{equation*}
        \iint_{\mathbb{R}_{p}^{d}\mathbb{R}_{k}^{d}} G(\mathbf{p},\mathbf{k})\delta(\omega(\mathbf{k};f)-\nabla_{p}E(\mathbf{p})\cdot\mathbf{k})(\nabla_{p}E(\mathbf{p})\cdot\mathbf{k}) = \iint_{\mathbb{R}_{p}^{d}\mathbb{R}_{k}^{d}} G(\mathbf{p},\mathbf{k})\delta(\omega(\mathbf{k};f)-\nabla_{p}E(\mathbf{p})\cdot\mathbf{k})\omega(\mathbf{k};f);
    \end{equation*} 
    Let $\eta = 1$ and $\phi = E(\mathbf{p})$, we obtain energy conservation
    \begin{equation*}
        \partial_{t}\left(\int_{\mathbb{R}_{p}^{d}} f(\mathbf{p},t) E(\mathbf{p}) + \int_{\mathbb{R}_{k}^{d}} W(\mathbf{k},t)\right)  = 0.
    \end{equation*}

Note that both equations share the same kernel $\Phi$, readers will find in  \cite{thorne2017modern} that $\Phi$ is actually the transition probability of a quantum process, in which particles emit and absorb plasmons(quantized waves) stochastically. Therefore, in some sense, the operators on the right hand side resemble collision operators, such as the Boltzmann collision operator and the Landau collision operator. For them, there have been a lot of works studying the well-posedness and regularity of solutions. For quasilinear theory, as far as we know, there is no rigorous analytical results. 

    The present work addresses the analytical construction and relaxation of solutions for the specific 1P1K model described below.

    \section{Weak formulation and main results}

In this paper we pursue a weak solution in the following sense,
    \begin{equation} \label{weak QLT 2}
        \left\{ \begin{aligned}-\int_{0}^{T}\left(f,\partial_{t}\zeta\right)_{p}dt-\left(f_{0},\zeta_{0}\right)_{p}&=\int_{0}^{T}\langle\nabla_{p}\cdot\Phi,fW\nabla_{p}\zeta\rangle_{pk}dt+\int_{0}^{T}\langle\Phi,fW\nabla_{p}\nabla_{p}\zeta\rangle_{pk}dt,\\
        -\int_{0}^{T}\left(W,\partial_{t}\xi\right)_{k}dt-\left(W_{0},\xi_{0}\right)_{k}&=-\int_{0}^{T}\langle\nabla_{p}\cdot\Phi,fW\xi\nabla_{p}E\rangle_{pk}dt-\int_{0}^{T}\langle\Phi,fW\xi\nabla_{p}\nabla_{p}E\rangle_{pk}dt.
        \end{aligned}
        \right.
    \end{equation} 

   This system is endowed with initial data $f_0(p)$ and $W_0(k)$ as defined in \eqref{eq_vedenov} and \eqref{eq_landau}.

    For scenarios of interest in physics, the Bohm-Gross dispersion relation is widely used to model warm plasmas,
    \begin{equation*}
        \omega(\mathbf{k};f)=\left(1+3\lambda_{D}^{2}(f)\mathbf{k}^{2}\right)^{1/2}\omega_{pe},
    \end{equation*}
    where the Debye length $\lambda_{D}$ is proportional to the thermal speed, i.e. the variance of particle \textit{pdf}, and the constant $\omega_{pe}$ is the plasma frequency.
    
    Assuming that the diffusion process rarely affects the particles' temperature, it is reasonable to write the following approximation,
    \begin{equation*}
        \omega(\mathbf{k})=\left(1+3\lambda_{D}^{2}(f_{0})\mathbf{k}^{2}\right)^{1/2}\omega_{pe}.
    \end{equation*}

    \bigskip

    The most important feature of the system lies in the fact that particles and waves interact through the absorption/emission kernel, which contains a Dirac delta. In physics this is called ``resonance", since particles with certain momentum only interact with waves with some particular wave vectors. The 1P1K model is special, because in this case each momentum $\mathbf{p}$ corresponds to only one wave vector $\mathbf{k}$, and vice versa. This is the reason that the following trick is feasible.

    Normalize all the quantities with time unit $\frac{1}{\omega_{pe}}$, mass unit $m_{e}$ and length unit $\frac{c}{\omega_{pe}}$, the 1P1K system reads,
    \begin{equation}\label{normalizedeq}
        \left\{ \begin{aligned}\partial_{t}f & =\partial_{p}\left(\left[\int_{\mathbb{R}_{k}}W\delta\left(\left(1+\lambda^{2}k^{2}\right)^{1/2}-pk\right)dk\right]\partial_{p}f\right),\\
    \partial_{t}W & =\left[\int_{\mathbb{R}_{p}}\left(\partial_{p}f\right)p\delta\left(\left(1+\lambda^{2}k^{2}\right)^{1/2}-pk\right)dp\right]W,
    \end{aligned}
    \right.
    \end{equation}
    where $\lambda = \frac{\sqrt{3}\omega_{pe}\lambda_{D}}{c}$.

We define the distributions in (\ref{weak QLT 2}) as follows,
\begin{equation}\label{def_distrib}
    \begin{split}
        \langle\delta\left(\left(1+\lambda^{2}k^{2}\right)^{1/2}-pk\right),y(p)z(k)\rangle_{pk}&\coloneqq\int_{(-\infty,-\lambda)\cup(\lambda,+\infty)}\frac{|p|}{p^{2}-\lambda^{2}}y(p)z\left(\frac{\text{sgn}(p)}{\sqrt{p^{2}-\lambda^{2}}}\right)dp,\\    \langle\partial_{p}\left(\delta\left(\left(1+\lambda^{2}k^{2}\right)^{1/2}-pk\right)\right),y(p)z(k)\rangle_{pk}&\coloneqq\int_{(-\infty,-\lambda)\cup(\lambda,+\infty)}y(p)\partial_{p}\left(\frac{|p|}{p^{2}-\lambda^{2}}z(\frac{\text{sgn}(p)}{\sqrt{p^{2}-\lambda^{2}}})\right)dp.
    \end{split}
\end{equation}

Note that the resonance condition,
\begin{equation*}
    \left(1+\lambda^{2}k^{2}\right)^{1/2}-pk=0,
\end{equation*}
is equivalent to
\begin{equation*}
    p=\text{sgn}(k)\left(k^{-2}+\lambda^{2}\right)^{1/2},
\end{equation*}
and
\begin{equation*}
    k=\text{sgn}(p)\left(p^{2}-\lambda^{2}\right)^{-1/2}.
\end{equation*}

\

Inspired by the work in \cite{ivanov1967quasilinear}, we introduce the  auxiliary function $u(p,t)$ as follows.

Let $u(p,t)\,:\left((-\infty,-\lambda)\cup (\lambda,\infty)\right)\times \mathbb{R}^+ \, \rightarrow\mathbb{R}^{+} \cup \{0\} ,$ be defined by 
\begin{equation}
    \label{def_auxiliary}
    u(p,t) \coloneqq \left(p^{2}-\lambda^{2}\right)^{-3/2}W\left(\text{sgn}(p)\left(p^{2}-\lambda^{2}\right)^{-1/2},t\right).  
\end{equation}

Without loss of generality, consider the positive half domain, i.e. $p\in (\lambda, +\infty)$. The kinetic system (\ref{normalizedeq}) take the  equivalent for to the new system 

\begin{equation}\label{eq_Ivanov}
    \left\{ 
    \begin{aligned}
    \partial_{t}f & =\partial_{p}\left(p\left(p^{2}-\lambda^{2}\right)^{1/2}u\partial_{p}f\right),\\
 \partial_{t}u & =p\left(p^{2}-\lambda^{2}\right)^{1/2}\left(\partial_{p}f\right)u.      
 \end{aligned}
 \right.
\end{equation}

Combining the kinetic system \eqref{normalizedeq} with the auxiliary function, whose initial $u(p,0):=\varphi_{0}(p)$, clearly
\begin{align}\label{f and u relation}
    \partial_{t}f&=\partial_{p}(\partial_{t}u), \nonumber\\[3pt]
   \textrm{therefore},
   \ \ f &:=\partial_{p}(u-\varphi_{0})+f_{0},
\end{align}
where $f_0=f_0(p,0)$ as defined in \eqref{eq_vedenov}, and $\varphi_{0}(p)$ defined at the top of this paragraph. 

Hence, substituting the above identity into the second row of (\ref{eq_Ivanov}), it follows that
\begin{equation*}
    \partial_{t}u=p\left(p^{2}-\lambda^{2}\right)^{1/2}\partial_{p}\left(\partial_{p}u-\partial_{p}\varphi_{0}+f_{0}\right)u(p,t)=\gamma u\partial_{p}^{2}u+g_{0}u,\\[3pt]
\end{equation*}
where $g_{0}(p) \coloneqq p\left(p^{2}-\lambda^{2}\right)^{1/2}\partial_{p}\left(f_{0}-\partial_{p}\varphi_{0}\right)$ and $\gamma(p)\coloneqq p\sqrt{p^2-\lambda^2}$.

The above equation can also be written in the following divergence form,
\begin{equation}\label{eq_pme}
    \partial_{t}u=\partial_{p}\left(\frac{\gamma}{2}\partial_{p}u^{2}\right)-\gamma\left(\partial_{p}u\right)^{2}-\frac{\partial_{p}\gamma}{2}\partial_{p}u^{2}+g_{0}u.
\end{equation}

\begin{remark}
    An interesting observation is that, the connection between porous medium equations and coupled kinetic systems not only appears here, but also in the Rosseland approximation for radiative transfer equations \cite{bardos1987rosseland}. However, we would like to point out that they are not of the same nature, since the Rosseland approximation is related to asymptotic behavior, while the decoupling trick here does not depend on any scaling parameter.
\end{remark}

Note that once we have a solution $u$ to the above equation, both particle \textit{pdf} $f$ and wave \textit{sed} $W$ are formally determined. Inspired by the above formal derivation, we present a rigorous statement in the following theorem.

Before presenting the main theorem, let us introduce the notations. Define the domain for auxiliary function as 
\begin{equation}\label{Omega_star}
    \begin{split}
        \Omega^{*} &\coloneqq (\lambda, +\infty),\\
        Q^{*}_{T} &\coloneqq \Omega^{*} \times (0,T).
    \end{split}
\end{equation}
Define the physical test spaces as
\begin{equation}\label{rev:physical-testspaces}
 \begin{aligned}
 \mathcal Z_T=\left\{ \zeta(p,t)\in C^{\infty}\left(\mathbb{R}_{p}\times[0,T]\right):\zeta(\cdot,t)\in C_{c}^{\infty}\left(\mathbb{R}_{p}\right)\ \forall t\in[0,T]\ \text{and}\ \zeta(\cdot,T)=0\right\}, \\
 \mathcal X_T=\left\{ \xi(k,t)\in C^{\infty}\left(\mathbb{R}_{k}\times[0,T]\right):\xi(\cdot,t)\in C_{c}^{\infty}\left(\mathbb{R}_{k}\right)\ \forall t\in[0,T]\ \text{and}\ \xi(\cdot,T)=0\right\}.
 \end{aligned}
\end{equation}

    \begin{theorem}[Global particle-wave solutions]\label{thm:mainA}
        Assume that the initial particle distribution $f_{0}(p)$ and wave spectrum $W_{0}(k)$ satisfy the following conditions:
        \begin{enumerate}[label=\textup{(H\arabic*)},leftmargin=*]
         \item $f_0\geq0$ on $\mathbb R$ and $f_0\in\mathcal S(\mathbb R)$.
         Explicitly, $f_0$ is smooth and
         \[
          \sup_{p\in\mathbb R}(1+|p|)^m|\partial_p^j f_0(p)|<\infty
          \qquad\text{for every pair of integers }m,j\geq0.
         \]
         \item $W_0\geq0$ on $\mathbb R$, $W_0\in C_{c}^\infty(\mathbb R)$, and there are constants $0<k_-<k_+<\infty$ such that
         \[
          \operatorname{supp}W_0\subset[k_-,k_+].
         \]
        \end{enumerate}
    Then we have
    \begin{enumerate}[label=\textup{(\roman*)},leftmargin=*]
    \item From the relation between $f(p,t)$ and $u(p,t)$ set in \eqref{f and u relation},  for $p\in\Omega^{*}$, define the following functions related to the initial data \eqref{eq_landau}, and \eqref{eq_vedenov} as follows,
    \begin{align}\label{Theo2 inital data}
            \varphi_{0}(p)&\coloneqq\left(p^{2}-\lambda^{2}\right)^{-\frac{3}{2}}W_{0}\left(\frac{1}{\sqrt{p^{2}-\lambda^{2}}}\right); \qquad \gamma(p) = p\sqrt{p^{2}-\lambda^{2}}.\\[4pt]
            h_{0}(p) &= f_{0}(p) - \varphi'_{0}(p), \quad
        \text{and}  
            \quad g_{0}(p)\coloneqq \gamma(p)h'_{0}(p). 
            \nonumber
    \end{align}
    There exists a nonnegative function $u$,
    \begin{equation*}
        u\in  L^q(0,T;W_0^{1,q}(\Omega^*)),\qquad 1\leq q<\infty,
    \end{equation*}
    such that
    \begin{equation}\label{eq_weak_pme}
        \begin{split}
            &-\left(u,\partial_{t}\eta\right)_{Q_{T}^{*}}+\left(\frac{\gamma}{2}\partial_{x}u^{2},\partial_{x}\eta\right)_{Q_{T}^{*}}\\
            =&\left(-\gamma(\partial_{x}u)^{2},\eta\right)_{Q_{T}^{*}}+\left(-\left(\frac{\partial_{x}\gamma}{2}\right)\partial_{x}u^{2},\eta\right)_{Q_{T}^{*}}+\left(g_{0}u,\eta\right)_{Q_{T}^{*}}+\left(\varphi_{0},\eta(x,0)\right)_{\Omega^{*}}.
        \end{split}
    \end{equation}
    for any
    \begin{equation*}
        \eta(x,t) \in \mathcal{V}\coloneqq\left\{ \eta\in C^{1}(Q_{T}^{*})\cap L^{\infty}(Q_{T}^{*}):\eta(\cdot,T)=0\right\} 
    \end{equation*}
    \item Use the auxiliary solution in (i) to construct the physical variables
    \begin{equation}\label{def_fW}
        \begin{split}
            W(k,t)&\coloneqq\begin{cases}
        \frac{1}{k^{3}}u\left(\frac{\left(1+\lambda^{2}k^{2}\right)^{1/2}}{k},t\right), & k\in\left(0,+\infty\right)\\
        0, & \text{otherwise}
        \end{cases}\\
        f(p,t)&\coloneqq f_{0}(p)+\begin{cases}
\partial_{p}u(p,t)-\partial_{p}\varphi_{0}(p), & p\in\left(\lambda,+\infty\right)\\
0, & \text{otherwise}
\end{cases}
        \end{split}
    \end{equation}
    
    The pair $(f,W)$ satisfies \eqref{weak QLT 2}, i.e.
    \begin{equation*}
        \left\{ \begin{aligned}-\int_{0}^{T}\left(f,\partial_{t}\zeta\right)_{p}dt-\left(f_{0},\zeta_{0}\right)_{p}&=\int_{0}^{T}\langle\nabla_{p}\cdot\Phi,fW\nabla_{p}\zeta\rangle_{pk}dt+\int_{0}^{T}\langle\Phi,fW\nabla_{p}\nabla_{p}\zeta\rangle_{pk}dt,\\
        -\int_{0}^{T}\left(W,\partial_{t}\xi\right)_{k}dt-\left(W_{0},\xi_{0}\right)_{k}&=-\int_{0}^{T}\langle\nabla_{p}\cdot\Phi,fW\xi\nabla_{p}E\rangle_{pk}dt-\int_{0}^{T}\langle\Phi,fW\xi\nabla_{p}\nabla_{p}E\rangle_{pk}dt,
        \end{aligned}
        \right.
    \end{equation*} 
    for any $(\zeta,\xi) \in \mathcal{Z}_{T}\times \mathcal{X}_{T}$. For any $t \in [0,T]$, both $f(\cdot,t)$ and $W(\cdot, t)$ are nonnegative a.e., with the following regularity
    \begin{equation*}
        \begin{aligned}
 f&\in L^q(0,T;L^q(\mathbb R)),\\
 W&\in L^q(0,T;W^{1,q}(\mathbb R)),
 \end{aligned}
 \qquad 1\leq q<\infty.
    \end{equation*}
    \end{enumerate}
    \end{theorem}


\bigskip
    In the above theorem, conclusion (i) will be proved in Section \ref{sec_existence}. In the following proposition, we show how auxiliary solution $u$ in (i) will lead to conclusion (ii), a pair $(f,W)$ that solves the particle-wave system.
    \begin{prop}
        Suppose that 
        \begin{equation*}
            u\in  L^q(0,T;W_0^{1,q}(\Omega^*)),\qquad 1\leq q<\infty
        \end{equation*}
        satisfy \eqref{eq_weak_pme}. Then the definitions \eqref{def_fW} give a solution of \eqref{weak QLT 2} for any $(\zeta,\xi) \in \mathcal{Z}_{T}\times \mathcal{X}_{T}$.
    \end{prop}

\begin{proof}
    By definition of $f(p,t)$ in (\ref{def_fW}), 
    \begin{equation*}
        \begin{split}
            &-\int_{0}^{T}\left(f,\partial_{t}\zeta\right)_{p}dt-\left(f_{0},\zeta_{0}\right)_{p}\\
            =&-\int_{0}^{T}\int_{\lambda}^{+\infty}\left[\left(\partial_{p}u-\partial_{p}\varphi_{0}\right)\partial_{t}\zeta\right]dpdt\\
            =&\int_{0}^{T}\int_{\lambda}^{+\infty}\left[\left(u-\varphi_{0}\right)\partial_{t}\partial_{p}\zeta\right]dpdt-\left.\int_{0}^{T}\left[\left(u-\varphi_{0}\right)\partial_{p}\zeta\right]dt\right|_{\lambda}^{+\infty}\\
            =&\int_{0}^{T}\int_{\lambda}^{+\infty}\left[\left(u-\varphi_{0}\right)\partial_{t}\partial_{p}\zeta\right]dpdt.\\
            =&\left(u,\partial_{t}\left(\partial_{p}\zeta\right)\right)_{Q_{T}^{*}}+\left(\varphi_{0},\partial_{p}\zeta_{0}\right)_{\Omega^{*}}\\
            =&\left(\frac{\gamma}{2}\partial_{p}u^{2},\partial_{p}^{2}\zeta\right)_{Q_{T}^{*}}+\left(\gamma(\partial_{p}u)^{2},\partial_{p}\zeta\right)_{Q_{T}^{*}}+\left(\left(\frac{\partial_{p}\gamma}{2}\right)\partial_{p}u^{2},\partial_{p}\zeta\right)_{Q_{T}^{*}}-\left(g_{0}u,\partial_{p}\zeta\right)_{Q_{T}^{*}}.
        \end{split}
    \end{equation*}

    Meanwhile, by definition of the distributions in (\ref{def_distrib}),
    \begin{equation*}
        \begin{split}
            &\int_{0}^{T}\langle\nabla_{p}\cdot\Phi,fW\nabla_{p}\zeta\rangle_{pk}dt+\int_{0}^{T}\langle\Phi,fW\nabla_{p}\nabla_{p}\zeta\rangle_{pk}dt\\
            =&\int_{0}^{T}\int_{+\lambda}^{+\infty}\left[f\partial_{p}\left(\frac{p}{p^{2}-\lambda^{2}}W(\frac{1}{\sqrt{p^{2}-\lambda^{2}}},t)\partial_{p}\zeta\right)\right]dpdt.
        \end{split}
    \end{equation*}

    Note that $W(\frac{1}{\sqrt{p^{2}-\lambda^{2}}},t)=\left(p^{2}-\lambda^{2}\right)^{\frac{3}{2}}u(p,t)$ on $(\lambda,+\infty)$, hence
    \begin{equation*}
        \begin{split}
            &\int_{0}^{T}\langle\nabla_{p}\cdot\Phi,fW\nabla_{p}\zeta\rangle_{pk}dt+\int_{0}^{T}\langle\Phi,fW\nabla_{p}\nabla_{p}\zeta\rangle_{pk}dt\\
            =&\int_{0}^{T}\int_{+\lambda}^{+\infty}\left[f\partial_{p}\left(p\sqrt{p^{2}-\lambda^{2}}u(p,t)\partial_{p}\zeta\right)\right]dpdt\\
            =&\int_{0}^{T}\int_{+\lambda}^{+\infty}\left[f\partial_{p}\left(\gamma(p)u(p,t)\partial_{p}\zeta\right)\right]dpdt.
        \end{split}
    \end{equation*}

Then it can be verified that since $\partial_{p}\zeta(p,t) \in \mathcal{V}$,
\begin{equation*}
    -\int_{0}^{T}\left(f,\partial_{t}\zeta\right)_{p}dt-\left(f_{0},\zeta_{0}\right)_{p}= \int_{0}^{T}\langle\nabla_{p}\cdot\Phi,fW\nabla_{p}\zeta\rangle_{pk}dt+\int_{0}^{T}\langle\Phi,fW\nabla_{p}\nabla_{p}\zeta\rangle_{pk}dt.
\end{equation*}

Next, consider the equation for $W(k,t)$. By definition of $W(k,t)$ in (\ref{def_fW}),
\begin{equation*}
    \begin{split}
        \int_{0}^{T}\int_{-\infty}^{+\infty}\left[W\partial_{t}\xi\right]dtdk &= \int_{0}^{T}\int_{0}^{+\infty}\left[\frac{1}{k^{3}}u\left(\frac{\sqrt{1+\lambda^{2}k^{2}}}{k},t\right)\partial_{t}\xi(k,t)\right]dkdt\\
        &=\int_{0}^{T}\int_{\lambda}^{+\infty}\left[u\left(p,t\right)\partial_{t}\left(p\xi\left(\frac{1}{\sqrt{p^{2}-\lambda^{2}}},t\right)\right)\right]dpdt.
    \end{split}
\end{equation*}

Meanwhile, by definition of the distributions in (\ref{def_distrib}),

\begin{equation*}
    \begin{split}
        &-\int_{0}^{T}\langle\nabla_{p}\cdot\Phi,fW\xi\nabla_{p}E\rangle_{pk}dt-\int_{0}^{T}\langle\Phi,fW\xi\nabla_{p}\nabla_{p}E\rangle_{pk}dt\\
        =&-\int_{0}^{T}\int_{\lambda}^{+\infty}\left[f(p)\partial_{p}\left(p\frac{p}{p^{2}-\lambda^{2}}W\left(\frac{1}{\sqrt{p^{2}-\lambda^{2}}},t\right)\xi\left(\frac{1}{\sqrt{p^{2}-\lambda^{2}}},t\right)\right)\right]dpdt\\
        =&-\int_{0}^{T}\int_{\lambda}^{+\infty}\left[f(p)\partial_{p}\left(\gamma u\left(p\xi\left(\frac{1}{\sqrt{p^{2}-\lambda^{2}}},t\right)\right)\right)\right]dpdt.
    \end{split}
\end{equation*}

    The uncut function $p\xi((p^2-\lambda^2)^{-1/2},t)$ need not lie in $\mathcal V_T$. Choose $\chi\in C_c^\infty(\Omega^*)$ equal to one on a neighborhood of $[a,b]$ and test instead with
\[
 \eta_\xi(p,t)=\chi(p)p\xi((p^2-\lambda^2)^{-1/2},t)\in\mathcal V_T.
\]
Since $u$, $\partial_pu$, and $\varphi_0$ vanish outside $[a,b]$, all terms involving derivatives of $\chi$ are zero. The preceding change-of-variable calculation therefore remains valid and gives

\begin{equation*}
    -\int_{0}^{T}\left(W,\partial_{t}\xi\right)_{k}dt-\left(W_{0},\xi_{0}\right)_{k}=-\int_{0}^{T}\langle\nabla_{p}\cdot\Phi,fW\xi\nabla_{p}E\rangle_{pk}dt-\int_{0}^{T}\langle\Phi,fW\xi\nabla_{p}\nabla_{p}E\rangle_{pk}dt.
\end{equation*}

\end{proof}

\begin{remark}
    It takes no extra effort to extend the above conclusion to the negative half domain $p\in(-\infty,-\lambda)$ by concatenating solutions. In that case, instead of requiring $\text{supp}\left(W(\cdot, t) \right) \subset (\varepsilon, \frac{1}{\varepsilon})$, we need $\text{supp}\left(W(\cdot, t) \right) \subset (-\frac{1}{\varepsilon}, -\varepsilon) \cup(\varepsilon, \frac{1}{\varepsilon})$.
\end{remark}

Note that in the rest of the article, to keep it consistent with existing literature in mathematics, we will use $x$ instead of $p$.

\subsection{The second main result: relaxation on one seeded band}
For Part~II assume, in addition to Theorem~\ref{thm:mainA}, that
\begin{equation}\label{lt:seeded}
 \{x:\varphi_0(x)>0\}=I=(a,b)\subset\subset \Omega* = (\lambda,\infty),
 \qquad \operatorname{supp}\varphi_0=[a,b].
\end{equation}
The globally smooth datum is flat at $a,b$. It is compactly supported in $\Omega^*$, but is \emph{not} in $C_c^\infty(I)$. For this result choose the approximation interval in Part~I to be exactly $I$.

Define the nonnegative cone as
\begin{equation}\label{eq:nonnega_cone}
 \mathcal K_I=\{v\in H_0^1(I):v\geq0\}.
\end{equation}
Denote the particle probability associated to auxiliary function $v$ as $f_v=h_0+v' = f_0 - \varphi_0' + v'$, and define the corresponding variational energy as
\begin{equation*}
    \mathcal{J}(v)=\frac12\|f_{v}\|_{L^2(I)}^2=\frac12\|h_0+v'\|_{L^2(I)}^2
\end{equation*}
This variational energy is distinct from the conserved physical energy $\int_{p} f(p)\frac{p^{2}}{2}$.

\begin{theorem}[Qualitative long-time relaxation]\label{lt:mainB}
Under \eqref{lt:seeded}, let $u$ be a global subsequential limit of the Part~I approximations posed on $I$. The variational energy $\mathcal J$ has a unique minimizer $u_\infty\in\mathcal K_I$. With $f_\infty=h_0+u_\infty'$ on $I$, as $t$ goes to infinity
\begin{equation}\label{lt:convergence-main}
 \begin{aligned}
 \mathcal J(u(t))&\longrightarrow\mathcal J(u_\infty),\\
 u(t)&\longrightarrow u_\infty\quad\text{strongly in }H_0^1(I),\\
 f(t)&\longrightarrow f_\infty\quad\text{strongly in }L^2(I)
 \end{aligned}
\end{equation}
The limiting particle density is non-increasing on $I$ and is constant on every connected component of $\{u_\infty>0\}$. Set $f_\infty=f_0$ off $I$, extend $u_\infty$ by zero, and define $W_\infty$ by \eqref{def_fW}. Then $f(t)-f_\infty\to0$ in $L^2(\mathbb R)$, and $W(t)\to W_\infty$ uniformly and in every finite $L^q$ on the wave band $K=(k(b),k(a))$. 
\end{theorem}

\begin{remark}
    No smoothness of $u_\infty$ and no finite relative entropy with reference $u_\infty$ are assumed for these qualitative conclusions. Every subsequential limit of this specified approximation family has the same limiting equilibrium; uniqueness of all weak trajectories is not claimed.
\end{remark}

\part{Construction of global weak solutions}

\section{Lifted extension of approximate solutions}\label{sec_extension}
The equation (\ref{eq_pme}) is difficult to tackle due to its degeneracy. Therefore we consider a series of approximate problems on a cut-off domain first. They are arbitrarily close to the original problem, but each one of them is strictly coercive, which ensures the existence of classical solutions. Furthermore, these approximate solutions have enough regularity, allowing us to test with various functions and to obtain bounds that are uniform in the parameter $n$.

    Choose a cut-off domain $\Omega=(a,b)$ such that $\lambda<a<b<\infty$ and $\operatorname{supp}\varphi_0\subset[a,b]$.

Consider the following approximate problem (\ref{eq_Sn}) of Equation(\ref{eq_pme}) on the cut-off domain $\Omega=(a, b)$,
\begin{equation}\label{eq_Sn}\tag{$\mathcal{S}_{n}$}
    \left\{ \begin{aligned}
    &\partial_{t}u_{n}=\gamma \mathcal{P}_{n}(u_{n})\partial_{x}^{2}u_{n}+g_{0}u_{n}, &\text{in }\Omega \times \mathbb{R}^{+}\\
    &u_{n}(x,t)=0, &\text{on }\partial\Omega\times \mathbb{R}^{+} \\
    &u_{n}(x,0)=\varphi_{0}(x),&\text{in }\Omega
    \end{aligned}\right.
\end{equation}
where $\mathcal{P}_{n}\in C^{\infty}(\mathbb{R})$ is a family of functions with the following properties,
\begin{enumerate}[label=(\roman*)]
    \item $\mathcal{P}_{n}(y)\ge\frac{1}{2n},\forall y\in\mathbb{R}.$
    \item $\mathcal{P}_{n}(y)=y+\frac{1}{n},\forall y\in\mathbb{R}^{+}.$
    \item $\mathcal{P}^{\prime}_{n}(y)\geq0.$
\end{enumerate}

The following maximum principle can be found in Theorem 2.1, Chapter I of Ladyzenskaja et al.'s book \cite{ladyzhenskaya1968linear}. Note that the bounds are independent of the parameter $n$.

\begin{theorem}{(maximum principle)}
If $u_{n}$ is a classical solution to the approximate problem (\ref{eq_Sn}), then $u_{n}(x,t)$ satisfies the maximum principle on $\overline{Q_{T}}$:
\begin{equation*}
    \left\{ \begin{aligned}
    &u_{n}(x,t) \geq 0\\
    &u_{n}(x,t) \leq max(\varphi_{0})\exp(\max(|g_{0}|)t)\\
    \end{aligned}\right.
\end{equation*}
\label{thm_max}
\end{theorem}

For the existence of classical solution to the approximate problems (\ref{eq_Sn}), we refer the readers to Theorem 6.1, Chapter V of Ladyzenskaja et al.'s book \cite{ladyzhenskaya1968linear}.
\begin{theorem}[\textbf{classical solution}]
For any $T>0$, the problem (\ref{eq_Sn}) admits an unique classical solution $u_{n}(x,t)$ on $\overline{Q_{T}}$. Moreover, $u_{n}(x,t)$ belongs to Hölder space $\mathcal{H}^{2+\beta,(2+\beta)/2}(\overline{Q_{T}}) \cap \mathcal{H}^{3+\beta,(3+\beta)/2}(Q_{T})$.
\label{thm_classical_sol}
\end{theorem}

\begin{remark}
    The Hölder space $\mathcal{H}^{2+\beta,(2+\beta)/2}(\overline{Q_{T}})$ consists of all the functions that are $\beta$-Hölder continuous in $\overline{Q_{T}}$, together with all derivatives of the form $D^{r}_{t}D^{s}_{x}$ for $2r+s<2$. Same for $\mathcal{H}^{3+\beta,(3+\beta)/2}(Q_{T})$.
\end{remark}

\begin{proof}

To begin with, write the equation in divergence form. 
    \begin{equation*}
        \begin{split}
            \mathcal{L}_{n}u\equiv&\partial_{t}u_{n}-\gamma \mathcal{P}_{n}(u_{n})\partial_{x}^{2}u_{n}-g_{0}u_{n}\\
            =&\partial_{t}u_{n}-\partial_{x}\left(\gamma \mathcal{P}_{n}(u_{n})\partial_{x}u_{n}\right)+\gamma \mathcal{P}_{n}^{\prime}(u_{n})\left(\partial_{x}u_{n}\right)^{2}+\left(\partial_{x}\gamma \right)\mathcal{P}_{n}(u_{n})\partial_{x}u_{n}-g_{0}u_{n}.
        \end{split}
    \end{equation*}

Define

    \begin{equation*}
    \varphi(x,t)\coloneqq\varphi_{0}(x)+\left[\gamma \mathcal{P}_{n}(\varphi_{0})\partial_{x}^{2}\varphi_{0}+g_{0}\varphi_{0}\right]t.
\end{equation*}

Since $\varphi_{0} \in C^{\infty}_{0}(\Omega)$, the initial-boundary condition can be written as 
\begin{equation*}
    u_{n}|_{\Gamma_{T}} = \varphi |_{\Gamma_{T}},
\end{equation*}
where $\Gamma_{T}=\partial\Omega\times[0,T]\cup\{(x,t):x\in\Omega,t=0\}$.

In accordance with the notation of Ladyzenskaja, define
\begin{equation*}
    \begin{split}
        a_{1}(x,t,u,p)&\coloneqq \gamma(x) \mathcal{P}_{n}(u)p,\\
        a(x,t,u,p)&\coloneqq \gamma(x) \mathcal{P}_{n}^{\prime}(u)p^{2}+ \partial_{x}\gamma (x)\mathcal{P}_{n}(u)
        p-g_{0}(x)u,\\
        A(x,t,u,p)&\coloneqq-g_{0}(x)u.
    \end{split}
\end{equation*}

Note that the $p$ variable here is just a notation from Ladyzenskaja representing $\partial_{x}u$, not particle momentum. 
The conditions in Ladyzenskaja's theorem can be easily verified. See Appendix.

\end{proof}

By Theorem \ref{thm_max}, $u_{n}$ is always non-negative, therefore by the second property of function $\mathcal{P}_{n}$, the problem (\ref{eq_Sn}) in divergence form is as follows,
\begin{equation}
    \label{divergence}
    \left\{ \begin{aligned}
    &\partial_{t}u_{n}=\partial_{x}\left(\gamma(u_{n}+\frac{1}{n})\partial_{x}u_{n}\right)-\gamma\left(\partial_{x}u_{n}\right)^{2}-\frac{\partial_{x}\gamma }{2}\partial_{x}(u_{n}+\frac{1}{n})^{2}+g_{0}u_{n}, & \text{in }\Omega\times \mathbb{R}^{+}\\
    &u_{n}(x)=0, &\text{on }\partial\Omega\times \mathbb{R}^{+} \\
    &u_{n}(x,0)=\varphi_{0}(x),&\text{in }\Omega
    \end{aligned}\right.
\end{equation}

Define the following continuous lifted extension of $u_{n}$ on $\Omega^{*}=(\lambda,+\infty)$:
\begin{equation}\label{def_extension}
    \tilde{u}_{n}(x,t)\coloneqq\begin{cases}
\frac{1}{n}, & x\in(\lambda,a)\\
u_{n}+\frac{1}{n}, & x\in\Omega=(a,b)\\
\frac{1}{n}h(x), & x\in(b,+\infty)
\end{cases}
\end{equation}
with the tail $h(x)$ s.t. $h(b)=1$ decaying fast enough.

In particular,
\begin{equation}\label{def_h}
    h(x) = e^{-(x-b)},
\end{equation}
is a suitable choice, in which case the derivative of lifted extension $\tilde{u}_{n}$ is:
\begin{equation}\label{eq:dx_u_tilde}
    \partial_{x}\tilde{u}_{n}(x,t)=\begin{cases}
0, & x\in(\lambda,a)\\
\partial_{x}u_{n}, & x\in\Omega=(a,b)\\
-\frac{1}{n}h(x), & x\in(b,+\infty)
\end{cases}
\end{equation}

Note that $\tilde{u}_{n}$ is strictly positive, and it solves the following problem on the cut-off domain $\Omega=(a,b)$,
\begin{equation}\label{tilde_eq}
    \left\{ \begin{aligned}
    &\partial_{t}\tilde{u}_{n}=\partial_{x}\left(\gamma\tilde{u}_{n}\partial_{x}\tilde{u}_{n}\right)-\gamma\left(\partial_{x}\tilde{u}_{n}\right)^{2}-\frac{\partial_{x}\gamma }{2}\partial_{x}\tilde{u}_{n}^{2}+g_{0}\left(\tilde{u}_{n}-\frac{1}{n}\right), & \text{in }\Omega\times \mathbb{R}^{+}\\
    &\tilde{u}_{n}(x,t)=\frac{1}{n}, &\text{on }\partial\Omega\times \mathbb{R}^{+} \\
    &\tilde{u}_{n}(x,0)=\varphi_{0}(x)+\frac{1}{n},& \text{in }\Omega
    \end{aligned}\right.
\end{equation}

Obviously, $\tilde{u}_{n}$ and $\partial_{x}\tilde{u}_{n}$ converge to zero uniformly on $\Omega^{*}-\Omega$. In the next section, we are going to prove the existence of a limit function $u$ on $\Omega = (a,b)$ through regularity estimates uniform in $n$.

\section{Regularity estimates on approximate solutions}\label{sec_regularity}
In order to prove a.e. convergence of the gradient term in the next section, several regularity estimates are of necessity. In the work of Abdellaoui, Peral and Walias \cite{abdellaoui2015some}, the authors used existing results for an elliptic-parabolic problem with measure data. However for the problem we are dealing with, the nonlinear source term is not ``measure data", thus it calls for a different approach.

To begin with, we introduce the following estimate on the spatial derivative of lifted extension $\tilde{u}_{n}$.

    According to Ladyzenskaja et al. \cite{ladyzhenskaya1968linear}, our assumption on the smoothness of data yields enough smoothness of solutions, even the third order derivative is H\"older continuous. And that allows us to study a parabolic equation for the first derivative $\partial_{x} u_{n}$, which renders the following regularity estimates.

\begin{prop}[\textbf{Uniform $W^{1,q}(\Omega^{*})$ estimate of the lifted extension}]\label{prop_grad_bound}
If $\tilde{u}_{n}$ are the lifted extension of classical solutions $u_{n}$ as defined in (\ref{def_extension}), then their gradient in space are uniformly bounded in $L^{q}(\Omega^{*})$,
\begin{equation*}
    \sup_{t\in[0,T]}\lVert\partial_{x}\tilde{u}_{n}\rVert_{L^{q}(\Omega^{*})}\leq C_{1}(\varphi_{0},g_{0},T,q),\ \forall q\in[1,+\infty).
\end{equation*}
Consequently,
\begin{equation}\label{Linfty_W1_bound}
    \lVert\tilde{u}_{n}\rVert_{L^{\infty}(0,T;W^{1,q}(\Omega^{*}))}\leq C_{2}(\varphi_{0},g_{0},T,q),\ \forall q\in[1,+\infty).
\end{equation}
\end{prop}

\begin{proof}

To begin with, consider the non-divergence form in (\ref{eq_Sn}),
\begin{equation}\label{nondiv}
    \partial_{t}u_{n}-\text{\ensuremath{\gamma}}\left(u_{n}+\frac{1}{n}\right)\partial_{x}^{2}u_{n}-g_{0}u_{n}=0.
\end{equation}

Note that $\partial_{x} \tilde{u}_{n} = \partial_{x} u_{n}$ in cut-off domain $\Omega=(a,b)$ as given in Equation(\ref{eq:dx_u_tilde}).

Let $z_{n} = \partial_x u_{n}$, then by Theorem \ref{thm_classical_sol}, $z_{n} \in \mathcal{H}^{1+\beta,(1+\beta)/2}(\overline{Q_{T}}) \cap \mathcal{H}^{2+\beta,(2+\beta)/2}(Q_{T})$. Taking first derivative on both sides of Equation(\ref{nondiv}), every term is still continuous. Therefore $z_{n}$ satisfies the following linear parabolic equation, if we regard $u_{n}$ as data:
\begin{equation*}
    \partial_{t}z_{n}-\partial_{x}\left(\gamma\left(u_{n}+\frac{1}{n}\right)\partial_{x}z_{n}\right)-g_{0}z_{n}-u_{n}\partial_{x}g_{0}=0.
\end{equation*}

In addition, from Equation(\ref{nondiv}) and the boundary condition for $u_{n}$, it can be derived that
\begin{equation*}
    \partial_{x}z_{n}=\partial_{x}^{2}u_{n}=\frac{\partial_{t}u_{n}-g_{0}u_{n}}{\gamma\left(u_{n}+\frac{1}{n}\right)}=0 \text{ on }\partial \Omega\times \mathbb{R}^{+}.
\end{equation*}

To summarize, $z_{n}$ satisfies the following equations,
\begin{equation*}
    \left\{ \begin{aligned}
    &\partial_{t}z_{n}-\partial_{x}\left(\gamma\left(u_{n}+\frac{1}{n}\right)\partial_{x}z_{n}\right)-g_{0}z_{n}-u_{n}\partial_{x}g_{0}=0, &\text{in } \Omega\times \mathbb{R}^{+}\\
    &\partial_{x}z_{n}=0, &\text{on }\partial\Omega\times \mathbb{R}^{+}\\
    &z_{n}(x,0)=\partial_{x}\varphi_{0}(x), &\text{in }\Omega
    \end{aligned}\right.
\end{equation*}

Test the equation with $z_{n}^{2l+1}$ and perform integration by parts to obtain
\begin{equation}\label{zn_norm}
    \left(\partial_{t}z_{n},z_{n}^{2l+1}\right)_{\Omega}+\left(\gamma \left(u_{n}+\frac{1}{n}\right)\partial_{x}z_{n},(2l+1)z_{n}^{2l}\partial_{x}z_{n}\right)_{\Omega}=\left(g_{0}z_{n},z_{n}^{2l+1}\right)_{\Omega}+\left(u_{n}\partial_{x}g_{0},z_{n}^{2l+1}\right)_{\Omega}.
\end{equation}

Since the second term on the left hand side is non-negative, the following inequality holds,

\begin{equation*}
    \frac{1}{2l+2}\frac{d}{dt}\left(\int_{\Omega}z_{n}^{2l+2}\right)\leq\int_{\Omega}g_{0}z_{n}^{2l+2}+\int_{\Omega}u_{n}\partial_{x}g_{0}z_{n}^{2l+1}.
\end{equation*}

Again, use the maximum principle for $u_{n}$ in Theorem \ref{thm_max}, and by the assumption on the data $g_{0}$,
\begin{equation*} \frac{d}{dt}\left(\int_{\Omega}z_{n}^{2l+2}\right)\leq C_{1}(\varphi_{0}, g_{0},T)\int_{\Omega}z_{n}^{2l+2}+C_{2}(\varphi_{0}, g_{0},T)\int_{\Omega}|z_{n}|^{2l+1}.
\end{equation*}

Let $I = \int_{\Omega}z_{n}^{2l+2}$, by H\"older's inequality, the above is equivalent to,
\begin{equation*}
    \frac{d}{dt}I\leq C_{1}I+C_{3}I^{\frac{2l+1}{2l+2}}.
\end{equation*}

Apply Young's inequality on the second term to obtain
\begin{equation*}
    \frac{d}{dt}(I+C_{5})\leq C_{4}\left(I+C_{5}\right).
\end{equation*}

Therefore by Gr\"onwall's lemma we have
\begin{equation*}
    I\leq\left(I_{0}+C_{5}\right)\exp\left(C_{4}t\right)-C_{5}\leq\left(I_{0}+C_{5}\right)\exp\left(C_{4}T\right)-C_{5}.
\end{equation*}

    Consequently,
\begin{equation*}
    \sup_{t\in[0,T]}\lVert\partial_{x}u_{n}\rVert_{L^{2l}(\Omega)} \leq C_{6}(\varphi_{0}, g_{0}, T, l),\ l = 1, 2, \cdots
\end{equation*}

Note that the cut-off domain $\Omega=(a,b)$  is bounded, therefore
\begin{equation*}
    \sup_{t\in[0,T]}\lVert\partial_{x}u_{n}\rVert_{L^{q}(\Omega)} \leq C_{6}(\varphi_{0}, g_{0}, T, q),\ \forall q \in [1,+\infty).
\end{equation*}

By definition we have,
\begin{equation*}
   \begin{split}\lVert\partial_{x}\tilde{u}_{n}(\cdot,t)\rVert_{L^{q}(\Omega^{*})}^{q} & =\int_{\Omega^{*}}\lvert\partial_{x}\tilde{u}_{n}(\cdot,t)\rvert^{q}dx\\
 & =\int_{0}^{a}0\cdot dx+\int_{\Omega}\lvert\partial_{x}u_{n}(\cdot,t)\rvert^{q}dx+\int_{b}^{\infty}\left\vert\frac{1}{n}\partial_{x}h\right\vert ^{q}dx\\
 & =\lVert\partial_{x}u_{n}(\cdot,t)\rVert_{L^{q}(\Omega)}^{q}+\int_{b}^{\infty}\left\vert\frac{1}{n}\partial_{x}h\right\vert ^{q}dx,
\end{split}
\end{equation*}
hence
\begin{equation*}
    \sup_{t\in[0,T]}\lVert\partial_{x}\tilde{u}_{n}\rVert_{L^{q}(\Omega^{*})}\leq C_{7}(\varphi_{0},g_{0},T,q),\ \forall q \in [1,+\infty).
\end{equation*}

In addition, it can be easily verified that
\begin{equation*}
    \sup_{t\in[0,T]}\lVert\tilde{u}_{n}\rVert_{L^{q}(\Omega^{*})}\leq C_{8}(\varphi_{0},g_{0},T,q),\ \forall q \in [1,+\infty).
\end{equation*}

The above two bounds lead to the following uniform estimate,
\begin{equation*}
     \lVert\tilde{u}_{n}\rVert_{L^{\infty}(0,T;W^{1,q}(\Omega^{*}))}\leq C_{9}(\varphi_{0},g_{0},T,q),\ \forall q \in [1,+\infty).
\end{equation*}

\end{proof}

\begin{coro}\label{coro:dx_u_square}
    If $\tilde{u}_{n}$ is the lifted extension of classical solution $u_{n}$ as defined in (\ref{def_extension}), then the following inequality holds,
    \begin{equation*}
        \lVert \partial_{x}\tilde{u}_{n}^{2}\rVert_{L^{2}(Q^{*}_{T})} \leq C(\varphi_{0}, g_{0}, T),
    \end{equation*}
    where the constant $C$ is independent of $n$.
\end{coro}

\begin{proof}
    The conclusion is a straight forward consequence of the maximum principle in Theorem \ref{thm_max} combined with Proposition \ref{prop_grad_bound}. 
\end{proof}

\begin{prop}[\textbf{Uniform $W^{1,2}(Q^{*}_{T})$ estimate of $\tilde{u}^{2}_{n}$}]\label{prop_bound_W12}
    The sequence $\tilde{u}_{n}^{2}$ is uniformly bounded in $W^{1,2}(Q^{*}_{T})$.
\end{prop}

\begin{proof}
To begin with, we estimate the $L^{2}(Q^{*}_{T})$ norm of the temporal derivative.

Take Equation(\ref{tilde_eq}) and test it with $\partial_{t}\tilde{u}_{n}^{2}$. Integrate by parts in $x$, and the trace integral vanishes because $\partial_{t}\tilde{u}_{n}^{2}=0$ on $\partial \Omega$, hence
\begin{equation*}
    \left(\partial_{t}\tilde{u}_{n},\partial_{t}\tilde{u}_{n}^{2}\right)_{Q_{T}}=-\left(\frac{\gamma}{2}\partial_{x}\tilde{u}_{n}^{2},\partial_{t}\partial_{x}\tilde{u}_{n}^{2}\right)_{Q_{T}}+\left(-\gamma\left(\partial_{x}\tilde{u}_{n}\right)^{2}-\frac{\partial_{x}\gamma }{2}\partial_{x}\tilde{u}_{n}^{2}+g_{0}\left(\tilde{u}_{n}-\frac{1}{n}\right),\partial_{t}\tilde{u}_{n}^{2}\right)_{Q_{T}}.
\end{equation*}

Applying the fundamental theorem of calculus with respect to $t\in (0, T)$ on the first term of the right hand side,
\begin{equation*}
    \begin{split}
        &\left(\partial_{t}\tilde{u}_{n},\partial_{t}\tilde{u}_{n}^{2}\right)_{Q_{T}}\\
        =&-\left(\frac{\gamma}{4},\left(\partial_{x}\tilde{u}_{n}^{2}(T)\right)^{2}\right)_{\Omega}+\left(\frac{\gamma}{4},\left(\partial_{x}\tilde{u}_{n}^{2}(0)\right)^{2}\right)_{\Omega}+\left(-\gamma\left(\partial_{x}\tilde{u}_{n}\right)^{2}-\frac{\partial_{x}\gamma }{2}\partial_{x}\tilde{u}_{n}^{2}+g_{0}\left(\tilde{u}_{n}-\frac{1}{n}\right),\partial_{t}\tilde{u}_{n}^{2}\right)_{Q_{T}}.
    \end{split}
\end{equation*}

The goal is to bound $\left(\partial_{t}\tilde{u}^{2}_{n},\partial_{t}\tilde{u}_{n}^{2}\right)_{Q_{T}}$, however the left hand side is $\left(\partial_{t}\tilde{u}_{n},\partial_{t}\tilde{u}_{n}^{2}\right)_{Q_{T}}$.

Note that by maximum principle, there exists some constant $C_{1}$ such that $\tilde{u}_{n} \in (\frac{1}{2n}, \frac{C_{1}}{2})$, therefore,
\begin{equation*}
    \left(\partial_{t}\tilde{u}_{n}^{2}\right)^{2}=4\tilde{u}_{n}^{2}\left(\partial_{t}\tilde{u}_{n}\right)^{2}\leq C_{1}\cdot2\tilde{u}_{n}\left(\partial_{t}\tilde{u}_{n}\right)^{2} = C_{1} \left(\partial_{t}\tilde{u}_{n}\right)\left(\partial_{t}\tilde{u}_{n}^{2}\right).
\end{equation*}

Integrate both sides on $Q_{T}$,

\begin{equation*}
\begin{split}\lVert\partial_{t}\tilde{u}_{n}^{2}\rVert_{L^{2}(Q_{T})}^{2}\leq & C_{1}\left(\partial_{t}\tilde{u}_{n},\partial_{t}\tilde{u}_{n}^{2}\right)_{Q_{T}}\\
= & C_{1}\left(-\left(\frac{\gamma}{4},\left(\partial_{x}\tilde{u}_{n}^{2}(T)\right)^{2}\right)_{\Omega}+\left(\frac{\gamma}{4},\left(\partial_{x}\tilde{u}_{n}^{2}(0)\right)^{2}\right)_{\Omega}\right)\\
 & +C_{1}\left(-\gamma\left(\partial_{x}\tilde{u}_{n}\right)^{2}-\frac{\partial_{x}\gamma }{2}\partial_{x}\tilde{u}_{n}^{2}+g_{0}\left(\tilde{u}_{n}-\frac{1}{n}\right),\partial_{t}\tilde{u}_{n}^{2}\right)_{Q_{T}}\\
\leq & C_{1}\left(-\left(\frac{\gamma}{4},\left(\partial_{x}\tilde{u}_{n}^{2}(T)\right)^{2}\right)_{\Omega}+\left(\frac{\gamma}{4},\left(\partial_{x}\tilde{u}_{n}^{2}(0)\right)^{2}\right)_{\Omega}\right)\\
 & +C_{1}\lVert-\gamma\left(\partial_{x}\tilde{u}_{n}\right)^{2}-\frac{\partial_{x}\gamma }{2}\partial_{x}\tilde{u}_{n}^{2}+g_{0}\left(\tilde{u}_{n}-\frac{1}{n}\right)\rVert_{L^{2}(Q_{T})}\cdot\lVert\partial_{t}\tilde{u}_{n}^{2}\rVert_{L^{2}(Q_{T})},
\end{split}
\end{equation*}
in which we used H\"older's inequality. Then use Young's inequality to bound the last term in the inequality above,

\begin{equation*}
    \begin{split} & \lVert-\gamma\left(\partial_{x}\tilde{u}_{n}\right)^{2}-\frac{\partial_{x}\gamma }{2}\partial_{x}\tilde{u}_{n}^{2}+g_{0}\left(\tilde{u}_{n}-\frac{1}{n}\right)\rVert_{L^{2}(Q_{T})}\cdot\lVert\partial_{t}\tilde{u}_{n}^{2}\rVert_{L^{2}(Q_{T})}\\
\leq & \frac{C_{1}}{2}\lVert-\gamma\left(\partial_{x}\tilde{u}_{n}\right)^{2}-\frac{\partial_{x}\gamma }{2}\partial_{x}\tilde{u}_{n}^{2}+g_{0}\left(\tilde{u}_{n}-\frac{1}{n}\right)\rVert_{L^{2}(Q_{T})}^{2}+\frac{1}{2C_{1}}\lVert\partial_{t}\tilde{u}_{n}^{2}\rVert_{L^{2}(Q_{T})}^{2}.
\end{split}
\end{equation*}
It follows that
\begin{equation*}
    \begin{split}
        \lVert\partial_{t}\tilde{u}_{n}^{2}\rVert_{L^{2}(Q_{T}^{*})}^{2}=\lVert\partial_{t}\tilde{u}_{n}^{2}\rVert_{L^{2}(Q_{T})}^{2}\leq& 2C_{1}\left(-\left(\frac{\gamma}{4},\left(\partial_{x}\tilde{u}_{n}^{2}(T)\right)^{2}\right)_{\Omega}+\left(\frac{\gamma}{4},\left(\partial_{x}\tilde{u}_{n}^{2}(0)\right)^{2}\right)_{\Omega}\right)\\
        &+C_{1}^{2}\left\lVert -\gamma\left(\partial_{x}\tilde{u}_{n}\right)^{2}-\frac{\partial_{x}\gamma }{2}\partial_{x}\tilde{u}_{n}^{2}+g_{0}\left(\tilde{u}_{n}-\frac{1}{n}\right)\right\rVert _{L^{2}(Q_{T})}^{2}.
    \end{split}
\end{equation*}

By Theorem \ref{thm_max} and Proposition \ref{prop_grad_bound}, the right hand side is uniformly bounded, hence $\lVert\partial_{t}\tilde{u}_{n}^{2}\rVert_{L^{2}(Q_{T})}$ is uniformly bounded. 

Meanwhile by Corollary \ref{coro:dx_u_square}, the spatial derivative $\lVert\partial_{x}\tilde{u}_{n}^{2}\rVert_{L^{2}(Q^{*}_{T})}$ is also uniformly bounded, thus the result follows.
\end{proof}

\bigskip
The following lemma shows that convergence a.e. combined with uniform boundedness implies strong convergence.
\begin{lemma}[\textbf{strong convergence from a.e. convergence}]
\label{lemma_ae_imply_strong}
    For a sequence $v_{n}$ that is uniformly bounded in $L^{4}(Q_{T})$, if $v_{n}$ converges to $v\in L^{4}(Q_{T})$ almost everywhere in $Q_{T}$, then $v_{n}$ converges to $v$ strongly in $L^{2}(Q_{T})$. 
\end{lemma}

\begin{proof}
    By Egorov's theorem, for any $\epsilon>0$, there exists a measurable set $S_{\epsilon} \subset Q_{T}$ such that $|S_{\epsilon}|\leq \epsilon$ and $v_{n}\rightarrow v$ uniformly in $Q_{T}\backslash S_{\epsilon}$. Therefore,
    \begin{equation*}
        \begin{split}
            \int_{Q_{T}}|v_{n}-v|^{2}
            =&\int_{S_{\epsilon}}|v_{n}-v|^{2}+\int_{Q_{T}\backslash S_{\epsilon}}|v_{n}-v|^{2}\\
            = & \int_{Q_{T}}|v_{n}-v|^{2}\chi_{S_{\epsilon}}+\int_{Q_{T}}|v_{n}-v|^{2}\chi^{2}_{Q_{T}\backslash S_{\epsilon}}.
        \end{split}
    \end{equation*}

    Apply H\"older's inequality on both terms,
    \begin{equation*}
        \begin{split}
            \int_{Q_{T}}|v_{n}-v|^{2}
            \leq&\lVert|v_{n}-v|^{2}\rVert_{L^{2}(Q_{T})}\lVert\chi_{S_{\epsilon}}\rVert_{L^{2}(Q_{T})}+\left(\sup_{(x,t)\in Q_{T}\backslash S_{\epsilon}}|v_{n}-v|^{2}\right)\lVert\chi_{Q_{T}\backslash S_{\epsilon}}\rVert_{L^{1}(Q_{T})}\\
            \leq & \lVert v_{n}-v\rVert_{L^{4}(Q_{T})}^{2}\epsilon+C_{1}\left(\sup_{(x,t)\in Q_{T}\backslash S_{\epsilon}}|v_{n}-v|^{2}\right).
        \end{split}
    \end{equation*}

   Since $v_{n} \rightarrow v$ uniformly in $Q_{T}\backslash S_{\epsilon}$, taking the limit on both sides,
    \begin{equation*}
        \limsup_{n\rightarrow\infty}\int_{Q_{T}}|v_{n}-v|^{2}\leq\lVert v_{n}-v\rVert_{L^{4}(Q_{T})}^{2}\epsilon.
    \end{equation*}

   As $\epsilon$ is arbitrary and $\lVert v_{n}-v\rVert_{L^{4}(Q_{T})}$ is uniformly bounded, we conclude that $v_{n}$ converges to $v$ strongly in $L^{2}(Q_{T})$.
    
\end{proof}

\begin{coro}\label{coro_limit_exist}
    For any given $q\in \mathbb{N}^{+}$, there exists a function $u \in L^{q}(0,T;W_{0}^{1,q}(\Omega^{*}))$ such that up to a subsequence,
    \begin{enumerate}
        \item $\tilde{u}_{n}$ converge to $u$ strongly in $L^{q}(Q_{T})$, for any $q \in \mathbb{N}^{+}$;
        \item $\tilde{u}_{n}$ converge to $u$ weakly in $\ensuremath{L^{q}(0,T;W_{0}^{1,q}(\Omega^{*}))}$, for any $q \in \mathbb{N}^{+}$;
        \item $\partial_{x}\tilde{u}_{n}$ converge to $\partial_{x}u$ weakly in
        $L^{2}(Q_{T})$.
    \end{enumerate}
\end{coro}

\begin{proof}
    By compactness, Proposition \ref{prop_bound_W12} implies that, for any $q<\infty$, up to a subsequence, $\tilde{u}_{n}^{2}$ converge to some function $y$ strongly in $L^{q}(Q_{T})$. Up to a subsequence of that subsequence, $\tilde{u}_{n}$ converge to $\sqrt{y}$ almost everywhere in $Q_{T}$,   Therefore by Lemma \ref{lemma_ae_imply_strong}, $\tilde{u}_{n}$ converge to $u=\sqrt{y}$ strongly in $L^{q}(Q_{T})$, for any $q<\infty$.

    With inequality(\ref{Linfty_W1_bound}), use Banach-Alaoglu Theorem, there exists a function $u \in L^{q}(0,T;W_{0}^{1,q}(\Omega^{*}))$ such that up to a subsequence, $\tilde{u}_{n}$ converge to $u$ weakly in $\ensuremath{L^{q}(0,T;W_{0}^{1,q}(\Omega^{*}))}$.

    Since $\partial_{x}\tilde{u}_{n}$ is uniformly bounded in $L^{2}(Q_{T})$, by Banach-Alaoglu Theorem, there exists a function $z \in L^{2}(Q_{T})$ such that up to a subsequence, $\partial_{x}\tilde{u}_{n}$ converges to $z$ weakly in $L^{2}(Q_{T})$. Since $\tilde{u}_{n}\rightarrow u$ a.e. in $Q_{T}$, it follows that $z=\partial_{x} u$ in the sense of distributions.

    By taking subsequence of a subsequence, there must be a function $u$ satisfying all the above conditions simultaneously.
\end{proof}

\bigskip
Now it remains to show that such a function $u$ is indeed a weak solution to the problem. That requires $\partial_{x} \tilde{u}_{n}$ converging to $\partial_{x} u$ almost everywhere in $Q_{T}$, which will be elaborated in Theorem \ref{thm_converge_ae} of the next section. The proof of Theorem \ref{thm_converge_ae} will need a control of $\partial_{x}\left(\tilde{u}_{n}^{1-\theta}\right)$, therefore we present an auxiliary estimate as follows. 

\begin{prop}[\textbf{Uniform weighted $L^{1}$ bound of $\partial_{x}\left(\tilde{u}_{n}^{1-\theta}\right)$}]
     Let $\psi\in C_{0}^{\infty}(Q_{T})$ be s.t. $\psi\geq0$ in $Q_{T}$, then the sequence $\gamma \psi\tilde{u}_{n}^{-\theta}|\partial_{x}\tilde{u}_{n}|$ is uniformly bounded in $L^{1}(Q_{T})$ for any $\theta\in(0,1/2)$.
    \label{prop_power}
\end{prop}

\begin{proof}
Let $\psi\in C_{0}^{\infty}(Q_{T})$ be s.t. $\psi\geq0$ in $Q_{T}$. Test Equation(\ref{tilde_eq}) with $\psi\tilde{u}_{n}^{-\delta}$, where $\delta\in(0,1)$. Since $\psi = 0$ on $\partial Q_{T}$, integrating by parts on $x$, the trace integral vanishes, it follows that,
\begin{equation*}
    \begin{split}
        &\left(\partial_{t}\tilde{u}_{n},\psi\tilde{u}_{n}^{-\delta}\right)_{Q_{T}}+\left(\gamma\tilde{u}_{n}\partial_{x}\tilde{u}_{n},\partial_{x}\left(\psi\tilde{u}_{n}^{-\delta}\right)\right)_{Q_{T}}\\
        =&\left(\partial_{t}\tilde{u}_{n},\psi\tilde{u}_{n}^{-\delta}\right)_{Q_{T}}+\left(\gamma\tilde{u}_{n}\partial_{x}\tilde{u}_{n},\tilde{u}_{n}^{-\delta}\left(\partial_{x}\psi\right)\right)_{Q_{T}}+\left(\gamma\tilde{u}_{n}\partial_{x}\tilde{u}_{n},\psi\left(\partial_{x}\tilde{u}_{n}^{-\delta}\right)\right)_{Q_{T}}\\
        =&\left(-\gamma\left(\partial_{x}\tilde{u}_{n}\right)^{2},\psi\tilde{u}_{n}^{-\delta}\right)_{Q_{T}}+\left(-\frac{\partial_{x}\gamma }{2}\partial_{x}\tilde{u}_{n}^{2},\psi\tilde{u}_{n}^{-\delta}\right)_{Q_{T}}+\left(g_{0}\left(\tilde{u}_{n}-\frac{1}{n}\right),\psi\tilde{u}_{n}^{-\delta}\right)_{Q_{T}}.
    \end{split}
\end{equation*}

Thus, simplifying and rearranging each inner product term in the equation above to obtain 
\begin{equation*}
    \begin{split}
        &\left(\tilde{u}_{n}^{-\delta}\partial_{t}\tilde{u}_{n},\psi\right)_{Q_{T}}+\left(\tilde{u}_{n}^{1-\delta}\partial_{x}\tilde{u}_{n},\gamma\partial_{x}\psi\right)_{Q_{T}}+(-\delta)\left(\tilde{u}_{n}^{-\delta}\left(\partial_{x}\tilde{u}_{n}\right)^{2},\gamma\psi\right)_{Q_{T}}\\
        =&-\left(\tilde{u}_{n}^{-\delta}\left(\partial_{x}\tilde{u}_{n}\right)^{2},\gamma\psi\right)_{Q_{T}}-\left(\tilde{u}_{n}^{1-\delta}\partial_{x}\tilde{u}_{n},\psi\partial_{x}\gamma\right)_{Q_{T}}+\left(\tilde{u}_{n}^{-\delta}\left(\tilde{u}_{n}-\frac{1}{n}\right),\psi g_{0}\right)_{Q_{T}}.
    \end{split}
\end{equation*}

Next, collecting the third term on the left hand side and the first term on the right hand side, we have

\begin{equation}\label{prop4-3}
    \begin{split}
        &\left(1-\delta\right)\left(\gamma\psi,\tilde{u}_{n}^{-\delta}\left(\partial_{x}\tilde{u}_{n}\right)^{2}\right)_{Q_{T}}\\
        =&-\frac{1}{1-\delta}\left(\partial_{t}\tilde{u}_{n}^{1-\delta},\psi\right)_{Q_{T}}-\frac{1}{2-\delta}\left(\partial_{x}\left(\gamma\psi\right),\partial_{x}\tilde{u}_{n}^{2-\delta}\right)_{Q_{T}}+\left(\psi g_{0},\tilde{u}_{n}^{1-\delta}\right)_{Q_{T}}-\left(\psi g_{0},\frac{1}{n}\tilde{u}_{n}^{-\delta}\right)_{Q_{T}}.
    \end{split}
\end{equation}

The first three terms on the right hand side are apparently bounded. Indeed $\tilde{u}_{n}$ satisfies maximum principle and,
\begin{equation*}
    \begin{split}
        \left(\partial_{t}\tilde{u}_{n}^{1-\delta},\psi\right)_{Q_{T}}&=-\left(\tilde{u}_{n}^{1-\delta},\partial_{t}\psi\right)_{Q_{T}}\\
        \left(\partial_{x}\left(\gamma\psi\right),\partial_{x}\tilde{u}_{n}^{2-\delta}\right)_{Q_{T}}&=-\left(\partial_{x}^{2}\left(\gamma\psi\right),\tilde{u}_{n}^{2-\delta}\right)_{Q_{T}}.
    \end{split}
\end{equation*}

In addition,  since $\tilde{u}_{n} = u_{n} + \frac{1}{n} \geq \frac{1}{n}$, the fourth term \eqref{prop4-3} to be estimated   
\begin{equation*}
    \frac{1}{n}\tilde{u}_{n}^{-\delta}\leq n^{\delta-1}=\frac{1}{n^{1-\delta}} \leq 1.
\end{equation*}

Therefore,  the left hand side of \eqref{prop4-3}
becomes  uniformly bounded,
\begin{equation}\label{logbound}
    \left(\gamma\psi,\tilde{u}_{n}^{-\delta}\left(\partial_{x}\tilde{u}_{n}\right)^{2}\right)_{Q_{T}}\leq C(\varphi_{0},g_{0},T,\psi,\delta).
 \end{equation}

Finally, by H\"older's inequality, the $L^{1}$ norm of the sequence $\gamma \psi\tilde{u}_{n}^{-\delta/2}|\partial_{x}\tilde{u}_{n}|$ can be bounded as follows,
\begin{equation*}
\begin{split}\lVert\gamma\psi\tilde{u}_{n}^{-\delta/2}\lvert\partial_{x}\tilde{u}_{n}\rvert\rVert_{L^{1}(Q_{T})}= & \left(\left(\gamma\psi\right)^{1/2},\left(\gamma\psi\tilde{u}_{n}^{-\delta}\left(\partial_{x}\tilde{u}_{n}\right)^{2}\right)^{1/2}\right)_{Q_{T}}\\
\leq & \lVert\left(\gamma\psi\right)^{1/2}\rVert_{L^{2}(Q_{T})}\lVert\left(\gamma\psi\tilde{u}_{n}^{-\delta}\left(\partial_{x}\tilde{u}_{n}\right)^{2}\right)^{1/2}\rVert_{L^{2}(Q_{T})}\\
= & \lVert\left(\gamma\psi\right)^{1/2}\rVert_{L^{2}(Q_{T})}\sqrt{\left(\gamma\psi,\tilde{u}_{n}^{-\delta}\left(\partial_{x}\tilde{u}_{n}\right)^{2}\right)_{Q_{T}}}.
\end{split}
\end{equation*}

And so, by inequality (\ref{logbound}), the sequence $\gamma \psi\tilde{u}_{n}^{-\theta}|\partial_{x}\tilde{u}_{n}|$ is uniformly bounded in $L^{1}(Q_{T})$ for any $\theta = \frac{\delta}{2} \in(0,\frac{1}{2})$.

\end{proof}

\subsection{Quadratic energy}\label{lt:energy-section}
The following additional estimates will give fixed-time compactness and connect Part~I with the long-time argument. On $\Omega=(a,b)$ set $\varepsilon_n=1/n$ and
\[
 h_0=f_0-\varphi_0',\qquad f_n=h_0+\partial_xu_n,\qquad A_n=\gamma(u_n+\varepsilon_n).
\]
Since $g_0=\gamma h_0'$, the approximation is exactly
\begin{equation}\label{lt:approx-flux}
 \partial_tu_n=A_n\partial_xf_n-\varepsilon_ng_0.
\end{equation}
In particular, the full expression on the right, not $A_n\partial_xf_n$ separately, vanishes at $a,b$, because $u_n$ is identically zero there for all time.

\begin{prop}[Approximate energy estimate]\label{lt:energy-prop}
Let
\[
 D_n(t)=\int_\Omega A_n|\partial_xf_n|^2\,dx,
 \qquad C_0=\frac12\int_\Omega\frac{|g_0|^2}{\gamma}\,dx.
\]
Then, for $0\leq s\leq t$,
\begin{equation}\label{lt:energy-integrated}
\frac12\|f_n(t)\|_{L^{2}(\Omega)}^2+\frac12\int_s^tD_n(r)\,dr
 \leq \frac12\|f_n(s)\|_{L^{2}(\Omega)}^2+C_0\varepsilon_n(t-s).
\end{equation}
\end{prop}
\begin{proof}
Differentiate only at the classical approximate level. Since $\partial_tf_n=\partial_x\partial_tu_n$ and $\partial_tu_n=0$ on the boundary, integration by parts gives
\begin{align*}
\frac{\partial}{\partial t}\left(\frac12\|f_n(t)\|_{L^{2}(\Omega)}^2\right)&=\int_\Omega f_n\partial_x\partial_tu_n\,dx
       =-\int_\Omega\partial_xf_n\partial_tu_n\,dx\\
      &=-D_n+\varepsilon_n\int_\Omega g_0\partial_xf_n\,dx.
\end{align*}
Young's inequality and $A_n\geq\gamma\varepsilon_n$ yield
\[
 \varepsilon_n\left|\int_\Omega g_0\partial_xf_n\,dx\right|
 \leq\frac12D_n+\frac{\varepsilon_n^2}{2}
       \int_\Omega\frac{|g_0|^2}{A_n}\,dx
 \leq\frac12D_n+C_0\varepsilon_n.
\]
Integrating proves \eqref{lt:energy-integrated}.
\end{proof}

\section{Convergence results}\label{sec_convergence}
The aim of this section is to prove that the sequence $\partial_{x}u_{n}$ converge to $\partial_{x}u$ a.e. in $Q_{T}$, where we have adopted the techniques in the work of Abdellaoui, Peral and Walias \cite{abdellaoui2015some}. The roadmap is as follows:
\begin{enumerate}
    \item Using Proposition \ref{prop_grad_bound} and Proposition \ref{prop_bound_W12} to prove Lemma \ref{lemma_sqr_convergence}.
    \item Theorem \ref{thm_square_converge} is a simple corollary of Lemma \ref{lemma_sqr_convergence}.
    \item Combining Theorem \ref{thm_square_converge} and Proposition \ref{prop_power} to prove Theorem \ref{thm_converge_ae}.
\end{enumerate}

\begin{lemma}
    If $\tilde{u}_{n}$ is the solution to Equation(\ref{tilde_eq}), then for any $ s \in (0,1)$
    \begin{equation*}
         \lim_{n\rightarrow\infty}\int_{Q_{T}}\left[\gamma\tilde{u}_{n}\left(\partial_{x}\left(\tilde{u}_{n}-u\right)\right)^{2}\right]^{s} = 0.
    \end{equation*}
    \label{lemma_sqr_convergence}
\end{lemma}

\begin{proof}
Recall that $u\in L^{2}\left(0,T;W_{0}^{1,2}(\Omega)\right)$, introduce the time–regularization of $u(x,t)$ by Landes and Mustonen \cite{landes1994parabolic},
\begin{equation*}
    u_{\nu}(x,t)=\exp(-\nu t)\varphi_{0}(x)+\nu\int_{0}^{t}\exp(-\nu(t-s))u(x,s)ds.
\end{equation*}

It is known that 
\begin{enumerate}
    \item $u_{\nu}(x,t)$ converge to $u(x,t)$ strongly in $L^{2}\left(0,T;W_{0}^{1,2}(\Omega)\right)$.
    \item $u_{\nu}$ is the solution of the following problem,
    \begin{equation}
    \left\{ \begin{aligned}
        &\frac{1}{\nu}\partial_{t}u_{\nu}+u_{\nu}=u\\
        &u_{\nu}(x,0)=\varphi_{0}(x)
        \end{aligned}\right.
        \label{u_nu_def}
    \end{equation}
\end{enumerate}

 Define a cut-off function $T_{\varepsilon}$ as
\begin{equation}
    T_{\varepsilon}(y)=\left\{ \begin{aligned}
        &y,&\  y \in (-\varepsilon, \varepsilon)\\
        &\text{sign}(y)\varepsilon,&\ \text{otherwise}
        \end{aligned}\right.
        \label{T_ep}
\end{equation}

And define a non-negative function $J_{\varepsilon}(y)$, such that $J^{\prime}_{\varepsilon}(y) = T_{\varepsilon}(y)$,
\begin{equation}
    J_{\varepsilon}(y)=\left\{ \begin{aligned}
        &-\varepsilon y-\frac{1}{2}\varepsilon^{2},&\  y \in (-\infty, -\varepsilon)\\
        &\frac{1}{2}y^{2},&\  y \in (-\varepsilon, \varepsilon)\\
        &\varepsilon y-\frac{1}{2}\varepsilon^{2},&\  y \in (\varepsilon, \infty)
        \end{aligned}\right.
        \label{J_ep}
\end{equation}

It takes two steps to prove that $\int_{Q_{T}}\left[\gamma\tilde{u}_{n}\left(\partial_{x}\left(\tilde{u}_{n}-u\right)\right)^{2}\right]^{s}$ converge to zero,
\begin{enumerate}
    \item prove that $\int_{Q_{T}}\left[\gamma\tilde{u}_{n}\left(\partial_{x}\left(\tilde{u}_{n}-u\right)\right)^{2}\right]^{s}\chi\{|u_{n}-u_{\nu}|\leq\varepsilon\}$ converge to zero.
    \item prove that $\int_{Q_{T}}\left[\gamma\tilde{u}_{n}\left(\partial_{x}\left(\tilde{u}_{n}-u\right)\right)^{2}\right]^{s}\chi\{|u_{n}-u_{\nu}| > \varepsilon\}$ converge to zero.
\end{enumerate}

    For the first step, it suffices to show that $\int_{Q_{T}}\left[\gamma\tilde{u}_{n}\left(\partial_{x}\left(\tilde{u}_{n}-u\right)\right)^{2}\right]\chi\{|u_{n}-u_{\nu}|\leq\varepsilon\}$ converges to zero. Indeed, denote $\gamma\tilde{u}_{n}\left(\partial_{x}\left(\tilde{u}_{n}-u\right)\right)^{2}$ as $X_{n}$, since $X_{n}\geq 0$, H\"older inequality gives
    \begin{equation*}
        \int_{\{|u_n-u_\nu|\leq\varepsilon\}}X_n^s
 \leq |Q_T|^{1-s}
       \left(\int_{\{|u_n-u_\nu|\leq\varepsilon\}}X_n\right)^s.
    \end{equation*}
    To bound $\int_{\{|u_n-u_\nu|\leq\varepsilon\}}X_n$, do the following decomposition

\begin{equation*}
\begin{split} & \int_{Q_{T}}\left[\gamma\tilde{u}_{n}\left(\partial_{x}\left(\tilde{u}_{n}-u\right)\right)^{2}\right]\chi\{|u_{n}-u_{\nu}|\leq\varepsilon\}\\
= & \int_{\{|u_{n}-u_{\nu}|\leq\varepsilon\}}\gamma\tilde{u}_{n}\left(\partial_{x}\left(\tilde{u}_{n}-u\right)\right)^{2}\\
= & \int_{\{|u_{n}-u_{\nu}|\leq\varepsilon\}}\gamma\tilde{u}_{n}\left(\partial_{x}\tilde{u}_{n}\right)\partial_{x}\left(\tilde{u}_{n}-u\right)-\int_{\{|u_{n}-u_{\nu}|\leq\varepsilon\}}\gamma\tilde{u}_{n}\left(\partial_{x}u\right)\partial_{x}\left(\tilde{u}_{n}-u\right)\\
= & \int_{\{|u_{n}-u_{\nu}|\leq\varepsilon\}}\gamma\tilde{u}_{n}\left(\partial_{x}\tilde{u}_{n}\right)\partial_{x}\left(\tilde{u}_{n}-u\right)\\
 & -\int_{Q_{T}}\left[\gamma\left(\tilde{u}_{n}\chi\{|u_{n}-u_{\nu}|\leq\varepsilon\}-u\chi\{|u-u_{\nu}|\leq\varepsilon\}\right)\left(\partial_{x}u\right)\partial_{x}\left(\tilde{u}_{n}-u\right)\right]\\
 & -\int_{Q_{T}}\left[\gamma u\chi\{|u-u_{\nu}|\leq\varepsilon\}\left(\partial_{x}u\right)\partial_{x}\left(\tilde{u}_{n}-u\right)\right]\\
= & A_{1}+A_{2}+A_{3}.
\end{split}
\end{equation*}

Start first from $A_{2}$ and $A_{3}$, as their estimates are relatively simple and straightforward.

Indeed, by H\"older's inequality and Corollary \ref{coro_limit_exist}, it follows that,
\begin{equation}
    \begin{split}
        A_{2}=&-\int_{Q_{T}}\left[\gamma\left(\tilde{u}_{n}\chi\{|u_{n}-u_{\nu}|\leq\varepsilon\}-u\chi\{|u-u_{\nu}|\leq\varepsilon\}\right)\left(\partial_{x}u\right)\partial_{x}\left(\tilde{u}_{n}-u\right)\right]\\
        \leq& C_{1}(\Omega,T)\lVert\tilde{u}_{n}\chi\{|u_{n}-u_{\nu}|\leq\varepsilon\}-u\chi\{|u-u_{\nu}|\leq\varepsilon\}\rVert_{L^{2}(Q_{T})}\lVert\left(\partial_{x}u\right)\partial_{x}\left(\tilde{u}_{n}-u\right)\rVert_{L^{2}(Q_{T})}\\
        \leq & C_{2}(\varphi_{0},g_{0},\Omega,T)\lVert\tilde{u}_{n}\chi\{|u_{n}-u_{\nu}|\leq\varepsilon\}-u\chi\{|u-u_{\nu}|\leq\varepsilon\}\rVert_{L^{2}(Q_{T})}.
    \end{split}
    \label{A2bound}
\end{equation}

Also, the following term will converge to zero,
\begin{equation}
    \begin{split}
        A_{3}=-\int_{Q_{T}}\gamma u\left(\partial_{x}u\right)\left(\partial_{x}\tilde{u}_{n}-\partial_{x}u\right)\chi\{|u-u_{\nu}|\leq\varepsilon\}.
    \end{split}
    \label{A3bound}
\end{equation}

\bigskip
It remains to bound $A_{1} = \int_{\{|u_{n}-u_{\nu}|\leq\varepsilon\}}\gamma \tilde{u}_{n}\left(\partial_{x}\tilde{u}_{n}\right)\partial_{x}\left(\tilde{u}_{n}-u\right)$. This estimate is performed by first testing Equation(\ref{tilde_eq}) with $T_{\varepsilon}(u_{n}-u_{\nu})$, where $T_{\varepsilon}$ is defined in Equation(\ref{T_ep}), to obtain,
\begin{equation*}
\begin{split} & \left(\partial_{t}\tilde{u}_{n},T_{\varepsilon}(u_{n}-u_{\nu})\right)_{Q_{T}}+\left(\gamma\tilde{u}_{n}\partial_{x}\tilde{u}_{n},\partial_{x}\left(T_{\varepsilon}(u_{n}-u_{\nu})\right)\right)_{Q_{T}}\\
= & \left(-\gamma\left(\partial_{x}\tilde{u}_{n}\right)^{2},T_{\varepsilon}(u_{n}-u_{\nu})\right)_{Q_{T}}+\left(-\frac{\partial_{x}\gamma }{2}\partial_{x}\tilde{u}_{n}^{2},T_{\varepsilon}(u_{n}-u_{\nu})\right)_{Q_{T}}+\left(g_{0}\left(\tilde{u}_{n}-\frac{1}{n}\right),T_{\varepsilon}(u_{n}-u_{\nu})\right)_{Q_{T}}.
\end{split}
\end{equation*}

Since $\lvert T_{\varepsilon}(u_{n}-u_{\nu}) \rvert \leq \epsilon$, the right hand side of the above equation can be bounded as follows,
\begin{equation*}
\text{RHS}\leq\varepsilon\left(\lVert\gamma\left(\partial_{x}\tilde{u}_{n}\right)^{2}\rVert_{L^{1}(Q_{T})}+\lVert\frac{\partial_{x}\gamma }{2}\partial_{x}\tilde{u}_{n}^{2}\rVert_{L^{1}(Q_{T})}+\lVert g_{0}\left(\tilde{u}_{n}-\frac{1}{n}\right)\rVert_{L^{1}(Q_{T})}\right),
\end{equation*}
with the first term uniformly bounded by Proposition \ref{prop_grad_bound}, the second one uniformly bounded by Corollary \ref{coro:dx_u_square}, and the last term by maximum principle. Consequently,
\begin{equation*}
    \left(\gamma\tilde{u}_{n}\partial_{x}\tilde{u}_{n},\partial_{x}\left(T_{\varepsilon}(u_{n}-u_{\nu})\right)\right)_{Q_{T}}\leq C_{1}(\varphi_{0},g_{0},\Omega,T)\varepsilon-\left(\partial_{t}\tilde{u}_{n},T_{\varepsilon}(u_{n}-u_{\nu})\right)_{Q_{T}}.
\end{equation*}

Since $u_{\nu}$ is a solution of Equation(\ref{u_nu_def}), $\partial_{t} u_{\nu}$ can be replaced with $\nu(u-u_{\nu})$,
\begin{equation*}
    \begin{split}
        \left(\partial_{t}\tilde{u}_{n},T_{\varepsilon}(u_{n}-u_{\nu})\right)_{Q_{T}} =& \left(\partial_{t}\left(u_{n}-u_{\nu}\right),T_{\varepsilon}(u_{n}-u_{\nu})\right)_{Q_{T}}+\left(\partial_{t}u_{\nu},T_{\varepsilon}(u_{n}-u_{\nu})\right)_{Q_{T}}\\
        =& \left(\partial_{t}\left(u_{n}-u_{\nu}\right),T_{\varepsilon}(u_{n}-u_{\nu})\right)_{Q_{T}}+\nu\left(\left(u-u_{\nu}\right),T_{\varepsilon}(u_{n}-u_{\nu})\right)_{Q_{T}}\\
        =& \left(1,\partial_{t}J_{\varepsilon}(u_{n}-u_{\nu})\right)_{Q_{T}}+\nu\left(\left(u-u_{\nu}\right),T_{\varepsilon}(u_{n}-u_{\nu})\right)_{Q_{T}}\\
        =& \left(1,J_{\varepsilon}(u_{n}(T)-u_{\nu}(T))\right)_{\Omega}-\left(1,J_{\varepsilon}(u_{n}(0)-u_{\nu}(0))\right)_{\Omega}+\nu\left(\left(u-u_{\nu}\right),T_{\varepsilon}(u_{n}-u_{\nu})\right)_{Q_{T}},
    \end{split}
\end{equation*}
in which $J_{\varepsilon}$ is defined in Equation(\ref{J_ep}) as the anti-derivative of $T_{\varepsilon}$. 

Each term on the right hand side is bounded from below. 

Indeed, by definition of $J_{\varepsilon}$,
\begin{equation}\label{posi_J_ep}
    \left(1,J_{\varepsilon}(u_{n}(T)-u_{\nu}(T))\right)_{\Omega}\geq0.
\end{equation}

Since $u_{n}$ and $u_{\nu}$ share the same initial condition, the second term is actually zero.
\begin{equation}\label{same_init}
    \left(1,J_{\varepsilon}(u_{n}(0)-u_{\nu}(0))\right)_{\Omega}=\left(1,J_{\varepsilon}(\varphi_{0}-\varphi_{0})\right)_{\Omega}=0.
\end{equation}

By monotonicity of $T_\varepsilon$, $(T_\varepsilon(A+B)-T_\varepsilon(B))A\geq0$ for all real $A,B$. With $A=u-u_\nu$ and $B=u_n-u$, this gives
\begin{equation}\label{eq:monotone}
 \nu\bigl(u-u_\nu,T_\varepsilon(u_n-u_\nu)\bigr)_{Q_T}
 \geq\nu\bigl(u-u_\nu,T_\varepsilon(u_n-u)\bigr)_{Q_T}.
\end{equation}

Therefore, combining inequalities (\ref{posi_J_ep}), (\ref{same_init}) and (\ref{eq:monotone}),
\begin{equation*}
    \begin{split}
        \left(\gamma\tilde{u}_{n}\partial_{x}\tilde{u}_{n},\partial_{x}\left(T_{\varepsilon}(u_{n}-u_{\nu})\right)\right)_{Q_{T}} 
        \leq &C_{1}(\varphi_{0},g_{0},\Omega,T)\varepsilon-\left(\partial_{t}\tilde{u}_{n},T_{\varepsilon}(u_{n}-u_{\nu})\right)_{Q_{T}}\\
        \leq & C_{1}(\varphi_{0},g_{0},\Omega,T)\varepsilon - \nu\left(\left(u-u_{\nu}\right),T_{\varepsilon}(u_{n}-u)\right)_{Q_{T}}.
    \end{split}
\end{equation*}

And it follows that,
\begin{equation}
    \begin{split}
        A_{1}
        &=\int_{\{|u_{n}-u_{\nu}|\leq\varepsilon\}}\gamma \tilde{u}_{n}\left(\partial_{x}\tilde{u}_{n}\right)\partial_{x}\left(u_{n}-u\right)\\
        &=\int_{\{|u_{n}-u_{\nu}|\leq\varepsilon\}}\gamma \tilde{u}_{n}\left(\partial_{x}\tilde{u}_{n}\right)\partial_{x}\left(u_{n}-u_{\nu}\right)+\int_{\{|u_{n}-u_{\nu}|\leq\varepsilon\}}\gamma \tilde{u}_{n}\left(\partial_{x}\tilde{u}_{n}\right)\partial_{x}\left(u_{\nu}-u\right)\\
        &=\int_{Q_{T}}\gamma \tilde{u}_{n}\left(\partial_{x}\tilde{u}_{n}\right)\partial_{x}\left(T_{\varepsilon}\left(u_{n}-u_{\nu}\right)\right)+\int_{\{|u_{n}-u_{\nu}|\leq\varepsilon\}}\gamma \tilde{u}_{n}\left(\partial_{x}\tilde{u}_{n}\right)\partial_{x}\left(u_{\nu}-u\right)\\
        &\leq C_{1}(\varphi_{0},g_{0},\Omega,T)\varepsilon-\nu\left(\left(u-u_{\nu}\right),T_{\varepsilon}(u_{n}-u)\right)_{Q_{T}}+\int_{\{|u_{n}-u_{\nu}|\leq\varepsilon\}}\gamma \tilde{u}_{n}\left(\partial_{x}\tilde{u}_{n}\right)\partial_{x}\left(u_{\nu}-u\right).
    \end{split}
    \label{A1bound}
\end{equation}

Putting together the inequalities (\ref{A1bound}), (\ref{A2bound}) and (\ref{A3bound}),
\begin{equation}
    \begin{split}
        &\int_{Q_{T}}\left[\gamma \tilde{u}_{n}\left(\partial_{x}\left(\tilde{u}_{n}-u\right)\right)^{2}\right]\chi\{|u_{n}-u_{\nu}|\leq\varepsilon\}\\
        =&A_{1}+A_{2}+A_{3}\\
        \leq & C_{1}(\varphi_{0},g_{0},\Omega,T)\varepsilon-\nu\left(\left(u-u_{\nu}\right),T_{\varepsilon}(u_{n}-u)\right)_{Q_{T}}+\int_{\{|u_{n}-u_{\nu}|\leq\varepsilon\}}\gamma \tilde{u}_{n}\left(\partial_{x}\tilde{u}_{n}\right)\partial_{x}\left(u_{\nu}-u\right)\\
        &+C_{2}(\varphi_{0},g_{0},\Omega,T)\lVert\tilde{u}_{n}\chi\{|u_{n}-u_{\nu}|\leq\varepsilon\}-u\chi\{|u-u_{\nu}|\leq\varepsilon\}\rVert_{L^{2}(Q_{T})}\\
        &-\int_{Q_{T}}\gamma u\left(\partial_{x}u\right)\left(\partial_{x}\tilde{u}_{n}-\partial_{x}u\right)\chi\{|u-u_{\nu}|\leq\varepsilon\}\\
        =&B_{1}(n,\nu,\varepsilon).
    \end{split}
    \label{first_s_bound}
\end{equation}

Since $\partial_{x} \tilde{u}_{n}$ converge to $\partial_{x}u$ weakly in $L^{2}(Q_{T})$, the last term of $B_{1}$ converges to zero as $n$ goes to infinity, therefore,
\begin{equation*}
    \lim_{\varepsilon\rightarrow0^{+}}\limsup_{\nu\rightarrow\infty}\limsup_{n\rightarrow\infty}B_{1}(n,\nu, \varepsilon) = 0.
\end{equation*}

\bigskip
For the second step, consider $\int_{Q_{T}}\left[\gamma \tilde{u}_{n}\left(\partial_{x}\left(\tilde{u}_{n}-u\right)\right)^{2}\right]^{s}\chi\{|u_{n}-u_{\nu}| > \varepsilon\}$, using H\"older's inequality, 
\begin{equation}\label{second_s_bound}
    \begin{split}
        & \int_{Q_{T}}\left[\gamma \tilde{u}_{n}
        \left(\partial_{x}\left(\tilde{u}_{n}-u\right)\right)^{2}\right]^{s}
        \chi\{|u_{n}-u_{\nu}| > \varepsilon\} \\
        \leq & \, C_{1}(\Omega, T) 
        \left\lVert 
            \left(\partial_{x}u_{n}-\partial_{x}u\right)^{2s} 
        \right\rVert_{L^{\rho'}(Q_{T})} 
        \Big( 
            \left\lVert 
                \chi\{|u_{n}-u_{\nu}|>\varepsilon\} 
                - \chi\{|u-u_{\nu}|>\varepsilon\} 
            \right\rVert_{L^{\rho}(Q_{T})} \\
        & + 
            \left\lVert 
                \chi\{|u-u_{\nu}|>\varepsilon\} 
            \right\rVert_{L^{\rho}(Q_{T})} 
        \Big) \\
        = & \, B_{2}(n,\nu,\varepsilon).
    \end{split}
\end{equation}

Taking the limit,
\begin{equation*}
    \lim_{\varepsilon\rightarrow0^{+}}\limsup_{\nu\rightarrow\infty}\limsup_{n\rightarrow\infty}B_{2}(n,\nu, \varepsilon) = 0.
\end{equation*}

To summarize,
\begin{equation*}
    0 \leq \int_{Q_{T}}\left[\gamma \tilde{u}_{n}\left(\partial_{x}\left(\tilde{u}_{n}-u\right)\right)^{2}\right]^{s} \leq B_{1}(n,\nu, \varepsilon) + B_{2}(n,\nu, \varepsilon),
\end{equation*}
where $B_{1}$ and $B_{2}$ are on the right hand side of Equation (\ref{first_s_bound}) and (\ref{second_s_bound}). Consequently,
\begin{equation*}
    \lim_{n\rightarrow\infty}\int_{Q_{T}}\left[\gamma \tilde{u}_{n}\left(\partial_{x}\left(\tilde{u}_{n}-u\right)\right)^{2}\right]^{s} = 0.
\end{equation*}

\end{proof}

\begin{theorem}The sequence $\partial_{x}\tilde{u}_{n}^{2}=\partial_{x}(u_{n}+\frac{1}{n})^{2}$ converge to $\partial_{x}u^{2}$ strongly in $L^{\sigma}(Q_{T})$ for all $\sigma \in (0,2)$.
\label{thm_square_converge}
\end{theorem}

\begin{proof}
    Note that 
    \begin{equation*}
        \begin{split}
            &\int_{Q_{T}}\lvert\partial_{x}\tilde{u}_{n}^{2}-\partial_{x}u^{2}\rvert^{2s}\\
            =&2^{2s}\int_{Q_{T}}\lvert\tilde{u}_{n}\partial_{x}\tilde{u}_{n}-u\partial_{x}u\rvert^{2s}\\
            =&2^{2s}\int_{Q_{T}}\lvert\tilde{u}_{n}\partial_{x}\tilde{u}_{n}-\tilde{u}_{n}\partial_{x}u+\tilde{u}_{n}\partial_{x}u-u\partial_{x}u\rvert^{2s}\\
            =&2^{2s}\int_{Q_{T}}\lvert\left(\tilde{u}_{n}\partial_{x}\tilde{u}_{n}-\tilde{u}_{n}\partial_{x}u\right)+\partial_{x}u\left(\tilde{u}_{n}-u\right)\rvert^{2s}\\
            \leq & C\int_{Q_{T}}\left(\lvert\tilde{u}_{n}\partial_{x}\tilde{u}_{n}-\tilde{u}_{n}\partial_{x}u\rvert^{2s}+\lvert\partial_{x}u\left(\tilde{u}_{n}-u\right)\rvert^{2s}\right).
        \end{split}
    \end{equation*}
    By Lemma \ref{lemma_sqr_convergence} and Corollary \ref{coro_limit_exist}, both terms converge to zero if $s \in (0, 1)$. Let $\sigma = 2s$, then $\sigma \in (0,2)$.
\end{proof}

\begin{theorem}\label{thm_converge_ae}

    Up to a subsequence, $\partial_{x}u_{n}$ converge to $\partial_{x}u$ a.e. in $Q_{T}$.

\end{theorem}

\begin{proof}
Let $\psi\in C_{0}^{\infty}(Q_{T})$ be s.t. $\psi\geq0$ in $Q_{T}$. To prove convergence a.e., it is sufficient to show that for some $\alpha \in (0,1)$,
\begin{equation*}
    \lim_{n\rightarrow\infty}\int_{Q_{T}}|\partial_{x}u_{n}-\partial_{x}u|^{\alpha}\psi=0.
\end{equation*}

Decompose the domain $Q_{T}$,
\begin{equation}
    \begin{split}
        \int_{Q_{T}}|\partial_{x}u_{n}-\partial_{x}u|^{\alpha}\psi=&\int_{\{u=0\}}|\partial_{x}u_{n}-\partial_{x}u|^{\alpha}\psi+\int_{\{u>0\}}|\partial_{x}u_{n}-\partial_{x}u|^{\alpha}\psi\\
        =& \int_{\{u=0\}}|\partial_{x}u_{n}|^{\alpha}\psi+\int_{\{0<u\leq\frac{1}{m}\}}|\partial_{x}u_{n}-\partial_{x}u|^{s}\psi+\int_{\{u>\frac{1}{m}\}}|\partial_{x}u_{n}-\partial_{x}u|^{s}\psi\\
        =& A_{1} + A_{2} + A_{3}.
    \end{split}
    \label{domaindecomp}
\end{equation}

Using H\"older's inequality to get the bound of $A_{2}$,
\begin{equation*}
    \begin{split}
        A_{2} &= \int_{\{0<u\leq\frac{1}{m}\}}|\partial_{x}u_{n}-\partial_{x}u|^{s}\psi\\
        &\leq \lVert|\partial_{x}u_{n}-\partial_{x}u|^{s}\psi\rVert_{L^{2/s}(Q_{T})}\lVert\chi_{\{0<u\leq\frac{1}{m}\}}\rVert_{L^{\frac{2}{2-s}}(Q_{T})}\\
        & \leq C \lVert\chi_{\{0<u\leq\frac{1}{m}\}}\rVert_{L^{\frac{2}{2-s}}(Q_{T})}.
    \end{split}
\end{equation*}

Note that $\lVert\chi_{\{0<u\leq\frac{1}{m}\}}\rVert_{L^{\frac{2}{2-s}}(Q_{T})}$ can be arbitrarily small.

Next, by Theorem \ref{thm_square_converge}, it is known that $\partial_{x}\tilde{u}^{2}_{n}\rightarrow\partial_{x}u^{2}$ strongly in $L^{\sigma}(Q_{T})$ for all $\sigma < 2$, therefore $A_{3}$ converges to zero, in fact, 
\begin{equation*}
    \begin{split}
        A_{3}&=\int_{\{u>\frac{1}{m}\}}\frac{1}{\lvert u\rvert^{s}}\lvert u\partial_{x}u_{n}-u\partial_{x}u\rvert^{s}\psi\\
        &=\int_{\{u>\frac{1}{m}\}}\frac{1}{\lvert u\rvert^{s}}\lvert\left(u-\tilde{u}_{n}\right)\partial_{x}\tilde{u}_{n}+\frac{1}{2}\left(\partial_{x}\tilde{u}_{n}^{2}-\partial_{x}u^{2}\right)\rvert^{s}\psi\\
        &\leq m^{s}\int_{Q_{T}}\lvert\left(u-\tilde{u}_{n}\right)\partial_{x}\tilde{u}_{n}+\frac{1}{2}\left(\partial_{x}\tilde{u}_{n}^{2}-\partial_{x}u^{2}\right)\rvert^{s}\psi,
    \end{split}
\end{equation*}
and the limit follows from
\begin{equation*}
    \limsup_{n\rightarrow\infty}A_{3}(n)\leq m^{s}\limsup_{n\rightarrow\infty}\int_{Q_{T}}\lvert\left(u-\tilde{u}_{n}\right)\partial_{x}\tilde{u}_{n}+\frac{1}{2}\left(\partial_{x}\tilde{u}_{n}^{2}-\partial_{x}u^{2}\right)\rvert^{s}\psi=0.
\end{equation*}

Considering $A_{1}$ of Equation \ref{domaindecomp}, since $u_{n} \rightarrow u$ strongly in $L^{q}(Q_{T})$, by Egorov's Lemma, for every $\epsilon > 0$, there exists a measurable set $E_{\epsilon}$ such that $\lvert E_{\epsilon}\rvert \leq \epsilon$ and $u_{n} \rightarrow u$ uniformly in $Q_{T} \backslash E_{\epsilon}$. Take the following decomposition,
\begin{equation*}
    A_{1}=  \int_{\{u=0\}}|\partial_{x}u_{n}|^{\alpha}\psi=\int_{\{u=0\}\cap E_{\epsilon}}|\partial_{x}u_{n}|^{\alpha}\psi+\int_{\{u=0\}\cap Q_{T}\backslash E_{\epsilon}}|\partial_{x}u_{n}|^{\alpha}\psi.
\end{equation*}
The first term is bounded through H\"older's inequality,
\begin{equation*}
    \begin{split}
        \int_{\{u=0\}\cap E_{\epsilon}}|\partial_{x}u_{n}|^{\alpha}\psi=&\int_{Q_{T}}|\partial_{x}u_{n}|^{\alpha}\psi\chi_{\{u=0\}\cap E_{\epsilon}}\\
        \leq&\int_{Q_{T}}|\partial_{x}u_{n}|^{\alpha}\psi\chi_{E_{\epsilon}}\\
        \leq& C\lVert|\partial_{x}u_{n}|^{\alpha}\rVert_{L^{1/\alpha}(Q_{T})}\lVert\chi_{E_{\epsilon}}\rVert_{L^{1/(1-\alpha)}(Q_{T})}\leq C\epsilon^{1-\alpha}.
    \end{split}
\end{equation*}
The second one uses the fact that for any $\mu > 0$, there exists $N$ such that $|u_{n}-u|=|u_{n}|<\mu$ for all $n>N$ and for all $x\in \{u=0\}\cap Q_{T}\backslash E_{\epsilon}$. In other words, for $n>N$, $\{u=0\}\cap Q_{T}\backslash E_{\epsilon}$ is a subset of $\{u_{n}\leq \mu\}\cap Q_{T}\backslash E_{\epsilon}$, hence the integral
\begin{equation*}
    \begin{split}
        &\int_{\{u=0\}\cap Q_{T}\backslash E_{\epsilon}}|\partial_{x}u_{n}|^{\alpha}\psi\\
        \leq&\int_{\{u_{n}\leq \mu\}\cap Q_{T}\backslash E_{\epsilon}}|\partial_{x}u_{n}|^{\alpha}\psi\\
        \leq&\left(\mu+\frac{1}{n}\right)^{\theta\alpha}\int_{\{u_{n}\leq \mu\}\cap Q_{T}\backslash E_{\epsilon}}\left(\frac{|\partial_{x}u_{n}|}{(u_{n}+\frac{1}{n})^{\theta}}\right)^{\alpha}\psi\\
        \leq & \left(\mu+\frac{1}{n}\right)^{\theta\alpha}\int_{Q_{T}}\left(\frac{|\partial_{x}u_{n}|}{(u_{n}+\frac{1}{n})^{\theta}}\right)^{\alpha}\psi.
    \end{split}
\end{equation*}
The boundedness of $\int_{Q_{T}}\left(\frac{|\partial_{x}u_{n}|}{(u_{n}+\frac{1}{n})^{\theta}}\right)^{\alpha}\psi$ is secured by Proposition \ref{prop_power}. The result follows from taking $\mu \rightarrow0$.

\end{proof}

\section{Existence of global weak solutions}\label{sec_existence}

Using Equation(\ref{zn_norm}) in another direction, we have the following lemma which is essential for the proof of Theorem \ref{thm:mainA}.
    \begin{lemma}\label{lemma_boundaryterm}
        If $u_{n}$ are classical solutions to the problem(\ref{divergence}) for $t\in[0,T]$, then the integral of $\partial_{x}u_{n}$ on $(0,T)$ is bounded as follows,
        \begin{equation*}
            \left\Vert\int_{0}^{T}\partial_{x}u_{n}(\cdot,t)dt\right\Vert_{L^{\infty}(\Omega)}\leq C(\varphi_{0},g_{0},T)\sqrt{n}.
        \end{equation*}
    \end{lemma}

    \begin{proof}
        Let
        \begin{equation*}
            y_{n}(x)=\int_{0}^{T}z_{n}(x,t)dt=\int_{0}^{T}\partial_{x}u_{n}(x,t)dt.
        \end{equation*}
        The goal is to bound $\left\Vert y_{n}(x)\right\Vert_{L^{\infty}(\Omega)}$.

        Note that by definition
        \begin{equation*}
            \int_{a}^{b}y_{n}(x)dx=\int_{0}^{T}\int_{a}^{b}\partial_{x}u_{n}(x,t)dxdt=\int_{0}^{T}\left(u_{n}(b,t)-u_{n}(a,t)\right)dt=0,
        \end{equation*}
        hence it is possible to use the Poincare's inequality if we can derive a bound for 
        \begin{equation*}
            \lVert\partial_{x}y_{n}\rVert_{L^{2}(\Omega)}=\left[\int_{a}^{b}\left(\partial_{x}y_{n}\right)^{2}dx\right]^{\frac{1}{2}}.
        \end{equation*}

        By H\"older's inequality,
        \begin{equation*}
            \left\vert\partial_{x}y_{n}(x)\right\vert^{2}=\left\vert\int_{0}^{T}\partial_{x}z_{n}(x,t)dt\right\vert^{2}\leq\left(\int_{0}^{T}|\partial_{x}z_{n}(x,t)|dt\right)^{2}\leq T\cdot\left(\int_{0}^{T}|\partial_{x}z_{n}(x,t)|^{2}dt\right).
        \end{equation*}

        Integrate on $\Omega$, the above inequality yields,
        \begin{equation}\label{ineq_bound_l2dyn}
            \lVert\partial_{x}y_{n}\rVert_{L^{2}(\Omega)}^{2}=\int_{a}^{b}|\partial_{x}y_{n}(x)|^{2}dx\leq T\cdot\left(\int_{a}^{b}\int_{0}^{T}|\partial_{x}z_{n}(x,t)|^{2}dtdx\right)=T\cdot\lVert\partial_{x}z_{n}\rVert_{L^{2}(Q_{T})}^{2}.
        \end{equation}

        To bound the right hand side, recall Equation(\ref{zn_norm}) with $l=0$, 
        \begin{equation}\label{eq_bound_zn}
            \left(\gamma (u_{n}+\frac{1}{n})\partial_{x}z_{n},\partial_{x}z_{n}\right)_{\Omega}=-\left(\partial_{t}z_{n},z_{n}\right)_{\Omega}+\left(g_{0}z_{n},z_{n}\right)_{\Omega}+\left(u_{n}\partial_{x}g_{0},z_{n}\right)_{\Omega}.
        \end{equation}

        Since $u_{n} \geq 0$ and $\gamma(x)\geq \gamma_{m}$ on $\Omega$, it follows that
        \begin{equation*}
            \gamma_{m}\frac{1}{n}\left(\partial_{x}z_{n},\partial_{x}z_{n}\right)_{\Omega}\leq\left(\gamma(u_{n}+\frac{1}{n})\partial_{x}z_{n},\partial_{x}z_{n}\right)_{\Omega}.
        \end{equation*}

        Integrate on $(0,T)$ to obtain,
        \begin{equation*}
            \left(\partial_{x}z_{n},\partial_{x}z_{n}\right)_{Q_{T}}\leq\frac{n}{\gamma_{m}}\left[-\frac{1}{2}\left(z_{n}(T),z_{n}(T)\right)_{\Omega}+\frac{1}{2}\left(z_{n}(0),z_{n}(0)\right)_{\Omega}+\left(g_{0}z_{n},z_{n}\right)_{Q_{T}}+\left(u_{n}\partial_{x}g_{0},z_{n}\right)_{Q_{T}}\right].
        \end{equation*}

        By Proposition \ref{prop_grad_bound} and the maximum principle, the above inequality yields,
        \begin{equation*}
            \left(\partial_{x}z_{n},\partial_{x}z_{n}\right)_{Q_{T}}\leq C_{1}(\varphi_{0},g_{0},T)n.
        \end{equation*}
        
        Combining the inequality above with inequality (\ref{ineq_bound_l2dyn}), we have
        \begin{equation*}
            \lVert\partial_{x}y_{n}\rVert_{L^{2}(\Omega)}^{2} \leq C_{1}(\varphi_{0},g_{0},T)Tn.
        \end{equation*}
        Since $y_{n}$ has zero mean, the Poincare's inequality renders
        \begin{equation*}
            \lVert y_{n}\rVert_{H^{1}(\Omega)}^{2}\leq C_{2}(\varphi_{0},g_{0},T)n.
        \end{equation*}
        Thus by Sobolev's inequality we have
        \begin{equation*}
            \lVert y_{n}\rVert_{L^{\infty}(\Omega)}\leq C(\varphi_{0},g_{0},T) \sqrt{n}.
        \end{equation*}
        
    \end{proof}

We are now ready to prove the main result of the paper, i.e. Theorem \ref{thm:mainA}.

\begin{figure}
    \begin{center}
        \begin{tikzpicture}
            \draw[thick] (0,0) -- (2,0) node[midway, below] {$T_{1}$};
            \draw[thick] (2,0) -- (5,0) node[midway, below] {$T_{2}$};
            \draw[thick,->] (5,0) -- (8,0)  node[anchor=north west] {$x$} node[midway, below] {$T_{3}$};
            \draw[thick,->] (0,0) -- (0,3)  node[anchor=north east] {$\tilde{u}_{n}$};
            \draw[thick,dashed] (2,0) -- (2,3);
            \filldraw[black] (2,0) circle (1pt) node[anchor=north]{$a$};
            \draw[thick,dashed] (5,0) -- (5,3);
            \filldraw[black] (5,0) circle (1pt) node[anchor=north]{$b$};
            \draw[blue,thick] (0,1) -- (2,1);
            \filldraw[black] (0,1) circle (1pt) node[anchor=east]{$\frac{1}{n}$};
            \draw[blue,thick] (2,1) .. controls (2.5, 1) and (2.5,2) .. (3,2);
            \draw[blue,thick] (3,2) .. controls (3.5, 2) and (4.5,1) .. (5,1);
            \draw[blue,thick] (5,1) .. controls (6,0.2) and (6,0.25) .. (8,0.2);
        \end{tikzpicture}
    \end{center}
    \caption{The lifted extension}
    \label{fig_sol}
\end{figure}

\begin{proof}

Note that any integral on $\Omega^{*}$ is a summation of three pieces, 
\begin{equation*}
\begin{split}\left(\tilde{u}_{n},\partial_{t}\eta\right)_{Q_{T}^{*}} & =\int_{0}^{T}\int_{0}^{a}\frac{1}{n}\left(\partial_{t}\eta\right)dxdt+\left(u_{n}+\frac{1}{n},\partial_{t}\eta\right)_{Q_{T}}+\int_{0}^{T}\int_{b}^{+\infty}\frac{1}{n}h\left(\partial_{t}\eta\right)dxdt,\\
\left(\gamma\tilde{u}_{n}\partial_{x}\tilde{u}_{n},\partial_{x}\eta\right)_{Q_{T}^{*}} & =0+\left(\gamma(u_{n}+\frac{1}{n})\partial_{x}u_{n},\partial_{x}\eta\right)_{Q_{T}}+\int_{0}^{T}\int_{b}^{+\infty}\gamma\left(\frac{1}{n}h\right)\left(\frac{1}{n}\partial_{x}h\right)\left(\partial_{x}\eta\right)dxdt,\\
\left(\gamma\left(\partial_{x}\tilde{u}_{n}\right)^{2},\eta\right)_{Q_{T}^{*}} & =0+\left(\gamma\left(\partial_{x}u_{n}\right)^{2},\eta\right)_{Q_{T}}+\int_{0}^{T}\int_{b}^{+\infty}\left[\gamma\left(\frac{1}{n}\partial_{x}h\right)^{2}\eta\right]dxdt,\\
\left(\frac{\partial_{x}\gamma }{2}\partial_{x}\tilde{u}_{n}^{2},\eta\right)_{Q_{T}^{*}} & =0+\left(\left(\partial_{x}\gamma \right)(u_{n}+\frac{1}{n})\partial_{x}u_{n},\eta\right)_{Q_{T}}+\int_{0}^{T}\int_{b}^{+\infty}\left[\frac{\partial_{x}\gamma }{2}\left(\frac{1}{n^{2}}\partial_{x}h^{2}\right)\eta\right]dxdt,\\
\left(g_{0}\tilde{u}_{n},\eta\right)_{Q_{T}^{*}} & =\int_{0}^{T}\int_{0}^{a}\left(\frac{1}{n}g_{0}\eta\right)dxdt+\left(g_{0}\left(u_{n}+\frac{1}{n}\right),\eta\right)_{Q_{T}}+\int_{0}^{T}\int_{b}^{+\infty}\left(g_{0}\frac{1}{n}h\eta\right)dxdt,\\
\left(\varphi_{0},\eta(x,0)\right)_{\Omega^{*}} & =0+\left(\varphi_{0},\eta(x,0)\right)_{\Omega}+0.
\end{split}    
\end{equation*}
Consequently,
\begin{equation}
    \begin{split} & -\left(\tilde{u}_{n},\partial_{t}\eta\right)_{Q_{T}^{*}}+\left(\gamma\tilde{u}_{n}\partial_{x}\tilde{u}_{n},\partial_{x}\eta\right)_{Q_{T}^{*}}+\left(\gamma\left(\partial_{x}\tilde{u}_{n}\right)^{2},\eta\right)_{Q_{T}^{*}}+\left(\frac{\partial_{x}\gamma}{2}\partial_{x}\tilde{u}_{n}^{2},\eta\right)_{Q_{T}^{*}}\\
    &-\left(g_{0}\tilde{u}_{n},\eta\right)_{Q_{T}^{*}}-\left(\varphi_{0},\eta(x,0)\right)_{\Omega^{*}}\\
= & T_{1}+T_{2}+T_{3},
\end{split}
\end{equation}
where we collect the terms according to our partition of the domain, as illustrated in Figure \ref{fig_sol}, 
\begin{equation*}
    \begin{split}T_{1}= & -\int_{0}^{T}\int_{0}^{a}\frac{1}{n}\left(\partial_{t}\eta\right)dxdt+0+0+0-\int_{0}^{T}\int_{0}^{a}\left(\frac{1}{n}g_{0}\eta\right)dxdt-0,\\
T_{2}= & -\left(u_{n}+\frac{1}{n},\partial_{t}\eta\right)_{Q_{T}}+\left(\gamma(u_{n}+\frac{1}{n})\partial_{x}u_{n},\partial_{x}\eta\right)_{Q_{T}}+\left(\gamma\left(\partial_{x}u_{n}\right)^{2},\eta\right)_{Q_{T}}\\
 & +\left(\left(\partial_{x}\gamma\right)(u_{n}+\frac{1}{n})\partial_{x}u_{n},\eta\right)_{Q_{T}}-\left(g_{0}\left(u_{n}+\frac{1}{n}\right),\eta\right)_{Q_{T}}-\left(\varphi_{0},\eta(x,0)\right)_{\Omega},\\
T_{3}= & -\int_{0}^{T}\int_{b}^{+\infty}\frac{1}{n}h\left(\partial_{t}\eta\right)dxdt+\int_{0}^{T}\int_{b}^{+\infty}\gamma\left(\frac{1}{n}h\right)\left(\frac{1}{n}\partial_{x}h\right)\left(\partial_{x}\eta\right)dxdt+\int_{0}^{T}\int_{b}^{+\infty}\left[\gamma\left(\frac{1}{n}\partial_{x}h\right)^{2}\eta\right]dxdt\\
 & +\int_{0}^{T}\int_{b}^{+\infty}\left[\frac{\partial_{x}\gamma}{2}\left(\frac{1}{n^{2}}\partial_{x}h^{2}\right)\eta\right]dxdt-\int_{0}^{T}\int_{b}^{+\infty}\left(g_{0}\frac{1}{n}h\eta\right)dxdt-0.
\end{split}
\end{equation*}

We claim that up to a subsequence,
\begin{equation*}
    \begin{split}
        \lim_{n \rightarrow \infty} \Big\{ 
        & -\left(\tilde{u}_{n}, \partial_{t}\eta \right)_{Q_{T}^{*}} 
        + \left(\gamma \tilde{u}_{n} \partial_{x} \tilde{u}_{n}, \partial_{x}\eta \right)_{Q_{T}^{*}} + \left(\gamma \left(\partial_{x} \tilde{u}_{n} \right)^{2}, \eta \right)_{Q_{T}^{*}} \\
        & + \left(\frac{\partial_{x} \gamma}{2} \partial_{x} \tilde{u}_{n}^{2}, \eta \right)_{Q_{T}^{*}} - \left(g_{0} \tilde{u}_{n}, \eta \right)_{Q_{T}^{*}} 
        - \left(\varphi_{0}, \eta(x,0) \right)_{\Omega^{*}} 
        \Big\} \\
        = 
        & -\left(u, \partial_{t} \eta \right)_{Q_{T}^{*}} 
        + \left(\gamma u \partial_{x} u, \partial_{x} \eta \right)_{Q_{T}^{*}}  + \left(\gamma \left(\partial_{x} u \right)^{2}, \eta \right)_{Q_{T}^{*}} 
        \\
        &+ \left(\frac{\partial_{x} \gamma}{2} \partial_{x} u^{2}, \eta \right)_{Q_{T}^{*}}  - \left(g_{0} u, \eta \right)_{Q_{T}^{*}} 
        - \left(\varphi_{0}, \eta(x,0) \right)_{\Omega^{*}}.
    \end{split}
\end{equation*}

The only non-trivial part is the convergence of the third term on the left hand side, for which we take the difference and use H\"older's inequality,
\begin{equation*}
\left(\gamma\left[\left(\partial_{x}\tilde{u}_{n}\right)^{2}-(\partial_{x}u)^{2}\right],\eta\right)_{Q_{T}}\leq C\lVert\partial_{x}\tilde{u}_{n}-\partial_{x}u\rVert_{L^{2}(Q_{T})}\lVert\partial_{x}\tilde{u}_{n}+\partial_{x}u\rVert_{L^{2}(Q_{T})}.
\end{equation*}

Since $\partial_{x}\tilde{u}_{n}$ converge to $\partial_{x}u$ a.e. in $Q_{T}$ and $\partial_{x}\tilde{u}_{n}$ is uniformly bounded in $L^{4}(Q_{T})$, by Lemma \ref{lemma_ae_imply_strong}, 
\begin{equation*}
    \lim_{n\rightarrow\infty} \lVert \partial_{x}\tilde{u}_{n} - \partial_{x}u \rVert_{L^{2}(Q_{T})} = 0.
\end{equation*}

Now it remains to prove that
\begin{equation*}
    \lim_{n\rightarrow +\infty} T_{1} + T_{2} +T_{3} = 0.
\end{equation*}

It can be easily verified that $T_{1}$ and $T_{3}$ goes to zero, given the definition of lifted extension $\tilde{u}_{n}$ in (\ref{def_extension}). For the second row, test Equation(\ref{divergence}) with $\eta$ and integrate by parts, it follows that 
\begin{equation*}
    T_{2}=-\left(\frac{1}{n},\partial_{t}\eta\right)_{Q_{T}}+\gamma(b)\int_{0}^{T}\frac{1}{n}\left(\partial_{x}u_{n}(b^{-},t)\right)\eta(b,t)dt-\gamma(a)\int_{0}^{T}\frac{1}{n}\left(\partial_{x}u_{n}(a^{+},t)\right)\eta(a,t)dt-\frac{1}{n}\left(g_{0},\eta\right)_{Q_{T}}.
\end{equation*}

By Lemma \ref{lemma_boundaryterm} the second and third term are in the order of $\mathcal{O}(\frac{1}{\sqrt{n}})$, and the other two terms are in the order of $\mathcal{O}(\frac{1}{n})$, therefore $T_{2}$ goes to zero.

To conclude, the limit function $u$ satisfies the following weak form of the equation,
\begin{equation*}
   -\left(u,\partial_{t}\eta\right)_{Q_{T}^{*}}+\left(\frac{\gamma}{2}\partial_{x}u^{2},\partial_{x}\eta\right)_{Q_{T}^{*}}=\left(-\gamma(\partial_{x}u)^{2},\eta\right)_{Q_{T}^{*}}+\left(-\left(\frac{\partial_{x}\gamma}{2}\right)\partial_{x}u^{2},\eta\right)_{Q_{T}^{*}}+\left(g_{0}u,\eta\right)_{Q_{T}^{*}}+\left(\varphi_{0},\eta(x,0)\right)_{\Omega^{*}}.
\end{equation*}

\end{proof}

\newpage

\part{Equilibrium selection and long-time relaxation}
\section{The variationally selected equilibrium}\label{sec_equilibrium}

\subsection{Stationary states and the selection problem}\label{subsec:equi_nonunique}
If $u_{\infty}(x)$ is a steady state for Equation \eqref{eq_pme}, then formally we have
\begin{equation*}
\gamma(x)u_{\infty}\partial_{x}^{2}u_{\infty}+g_{0}u_{\infty}=0, \forall x\in(\lambda,+\infty).
\end{equation*}
Besides the trivial solution $u_{\infty}(x)\equiv0$, the following equation
\begin{equation*}
\gamma(x)\partial_{x}^{2}\mathcal{M}+g_{0} = \gamma(x)\partial_{x}^{2}\mathcal{M}+ \gamma(x)\partial_{x}\left(f_{0}-\partial_{x}\varphi_{0}\right) = 0,
\end{equation*}
renders a family of candidate equilibrium solutions:
\begin{equation}\label{eq_calM}
\mathcal{M}(x)=\varphi_{0}(x)-\int_{\lambda}^{x}f_{0}(s)ds+C_{1}x+C_{2}.
\end{equation}
Recall the relation  $h_{0}(x) = f_{0}(x) - \partial_{x} \varphi_{0}(x)$, as defined in \eqref{Theo2 inital data}. It can also be written as
\begin{equation*}
    \mathcal{M}(x)=-\int_{\lambda}^{x}h_{0}(s)ds+C_{1}x+C_{2},
\end{equation*}
with both $C_1$ and $C_2$ being arbitrary real valued constants.

The family \eqref{eq_calM} has a geometric interpretation. Fix a
reference point $x_*$ and write
\[
 H(x)=\int_{x_*}^{x}h_0(s)\,ds,
 \qquad \ell(x)=C_1x+C_2.
\]
After absorbing the choice of integration constant into $C_2$, the
candidate profile becomes $\mathcal M=\ell-H$. Thus $\mathcal M(x)$
is the vertical distance between an affine line and the primitive
curve $H$, wherever the line lies above the curve. Changing $C_1$
changes the slope of the line, while changing $C_2$ translates it
vertically.

Suppose that the line meets the curve at two interior points
$\alpha<\beta$ and lies strictly above it on $(\alpha,\beta)$.
The corresponding branch can be joined to zero by setting
\[
 u_\ell(x)=
 \begin{cases}
  \ell(x)-H(x),&\alpha<x<\beta,\\
  0,&\text{otherwise}.
 \end{cases}
\]
This profile is continuous and nonnegative. On its positive interval,
$u_\ell''=-h_0'$, so it satisfies the stationary equation there.
The associated particle density is
\[
 f_\ell=h_0+u_\ell'
 =\begin{cases}
   C_1,&\alpha<x<\beta,\\
   h_0(x),&x\notin[\alpha,\beta],
  \end{cases}
 \qquad\text{a.e.}
\]
Hence the slope of the line is precisely the height of a particle
plateau. Since the line passes through both contact points,
\[
 C_1=\frac{H(\beta)-H(\alpha)}{\beta-\alpha}
     =\frac{1}{\beta-\alpha}\int_\alpha^\beta h_0(s)\,ds.
\]

At a transverse intersection, joining the branch to zero creates a
corner. Indeed, the derivatives from within the positive interval are
$u_\ell'(\alpha^+)=C_1-h_0(\alpha)$ and
$u_\ell'(\beta^-)=C_1-h_0(\beta)$, whereas the derivative outside
is zero. Requiring $C^1$ matching at both endpoints therefore gives
\begin{equation}\label{eq:bitangent_matching}
 \ell(\alpha)=H(\alpha),\qquad
 \ell(\beta)=H(\beta),\qquad
 C_1=H'(\alpha)=H'(\beta).
\end{equation}
Geometrically, the line is tangent to the primitive curve at both
contact points: it is a \emph{bitangent}. The same conditions make
$f_\ell$ continuous at the edges of its plateau. Equivalently, the
plateau level satisfies
\[
 h_0(\alpha)=h_0(\beta)=C_1,
 \qquad
 \int_\alpha^\beta\bigl(h_0(x)-C_1\bigr)\,dx=0.
\]
The integral identity expresses equality of the total areas above
and below the plateau level. Figure~\ref{fig:bitangent_geometry}
compares transverse intersections with bitangent contact.

Here the regularity requirement is the matching of first derivatives
at interior free boundaries. Profiles with corners are still
Lipschitz and belong to $H^1$, and the degeneracy at $u=0$ permits
such junctions in the stationary weak formulation. Consequently,
nonnegativity and $H^1$ regularity alone do not select the bitangent
configuration.

Bitangency is also a local condition: a primitive
curve may admit several common tangents. The variational principle
introduced below resolves the global selection problem. As proved
in Proposition~\ref{lt:lcm-equilibrium}, it selects
$F_\infty=H+u_\infty$ as the least concave majorant of $H$ on the
prescribed interval. Its affine segments give the particle plateaus,
and a segment with two interior contact points is bitangent to $H$.
This provides a geometric interpretation of the variationally
selected equilibrium.

\begin{figure}[htbp]
 \centering
 \includegraphics[width=\linewidth]{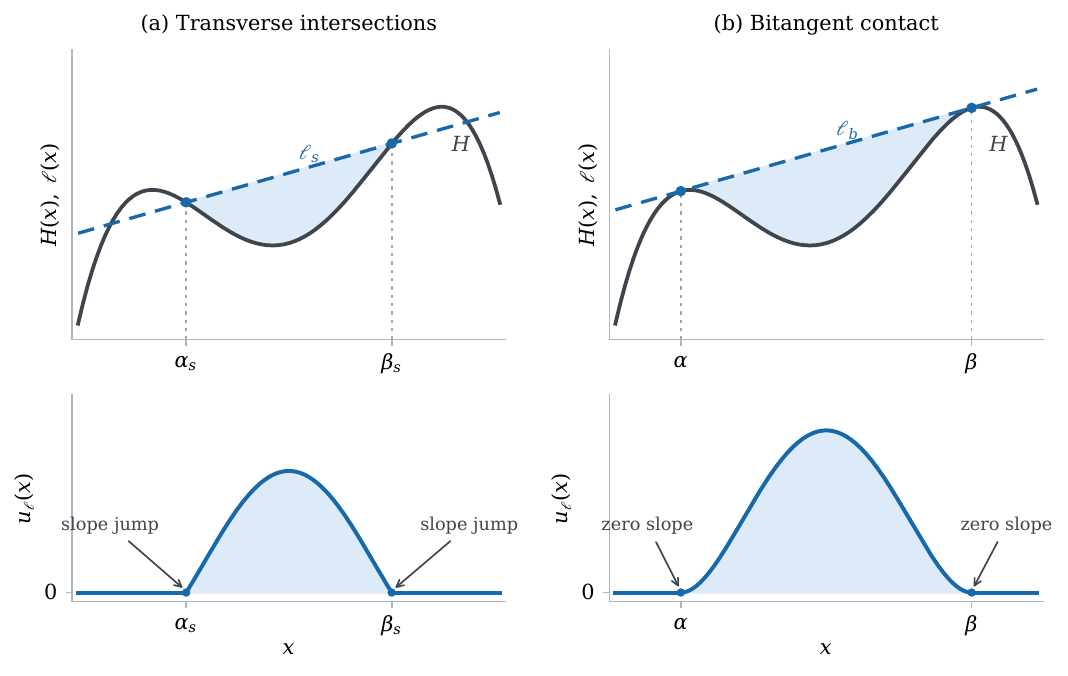}
 \caption{Geometric construction of a stationary branch (schematic).
 The upper panels show the primitive $H$ (solid gray), an affine line
 $\ell$ (dashed blue), and their gap on the selected interval
 (shaded). The lower panels show this gap joined to zero outside
 that interval. Transverse intersections produce jumps in the first
 derivative, while bitangent contact gives zero slope at both
 junctions and hence a $C^1$ profile.}
 \label{fig:bitangent_geometry}
\end{figure}

\subsection{The selected equilibrium on the fixed active band}
Formally, the particle probability density function $f(x,t)$ satisfies a diffusion equation $f_{t} = \partial_{x}\left(\gamma u\partial_{x}f\right)$ and we have $\partial_{t}\frac{1}{2}\lVert f\rVert^{2}_{L^{2}(\mathbb{R})} = - \left(\gamma u \partial_{x}f, \partial_{x}f\right) \leq 0$, therefore we expect the equilibrium to minimize the $L^{2}$ norm of $f$. The preceding formal family
does not decide which stationary state is reached. We select one by minimizing the variational energy over the nonnegative cone. This is an equilibrium-selection statement, not a uniqueness assertion for all weak stationary solutions.

For simplicity, consider the interval $I=(a,b)$ and an initial condition $u(x,0) = \varphi_{0}(x)$ satisfying the active band spectrum condition:
\begin{equation}
 \{x:\varphi_0(x)>0\}=I=(a,b)\Subset(\lambda,\infty),
 \qquad \operatorname{supp}\varphi_0=[a,b].
\end{equation}
In physical language, we require the initial wave spectrum density to be non-zero inside a fixed range of wave numbers.

Recall that the particle probability density $f$ associated to an auxiliary function $v$ is $f_{v} = v' + f_{0} - \varphi'_{0} = v' + h_{0}$. Let us define the variational energy associated to $v$ as 
\begin{equation*}
    \mathcal{J}(v)=\frac12\|f_{v}\|_{L^2(I)}^2=\frac12\|h_0+v'\|_{L^2(I)}^2.
\end{equation*}
In what follows, we select one stationary state $u_{\infty}$ by minimizing the variational energy within the the non-negative cone $\mathcal{K}_I=\{v\in H_0^1(I):v\geq0\}$.

\begin{prop}[Existence and uniqueness of the minimizer]\label{lt:minimizer}
There is a unique $u_\infty\in\mathcal K_I$ such that
\[
 \mathcal J(u_\infty)=\inf_{v\in\mathcal K_I}\mathcal J(v).
\]
Put $f_\infty=h_0+u_\infty'$. Then, for every $v\in\mathcal K_I$,
\begin{equation}\label{lt:VI}
 \int_I f_\infty(v'-u_\infty')\,dx\geq0,
\end{equation}
and
\begin{equation}\label{lt:gap}
 \mathcal J(v)-\mathcal J(u_\infty)
 \geq\frac12\|v'-u_\infty'\|_2^2
 =\frac12\|f_v-f_\infty\|_2^2.
\end{equation}
\end{prop}

\begin{proof}
    By definition
    \begin{equation*}
        0\leq \inf_{v\in\mathcal K_I}\mathcal J(v) \leq J(0) = \frac12\|h_{0}\|_{L^2(I)}^2 < \infty.
    \end{equation*}
    For each integer $j\ge 1$, choose $v_{j} \in \mathcal{K}_{I}$ such that
    \begin{equation*}
        \inf_{v\in\mathcal K_I}\mathcal J(v) \leq \mathcal J(v_{j}) \leq \inf_{v\in\mathcal K_I}\mathcal J(v) + \frac{1}{j}.
    \end{equation*}
    This minimizing sequence $v_{j}$ is bounded in $H^{1}_{0}(I)$. Indeed, the triangle inequality gives a bound for the derivatives:
    \[
     \|v_j'\|_2
     =\|(h_0+v_j')-h_0\|_2
     \leq\|h_0+v_j'\|_2+\|h_0\|_2
     =\sqrt{2\mathcal J(v_j)}+\|h_0\|_2
     \leq\sqrt{2 \inf_{v\in\mathcal K_I}\mathcal J(v) +2}+\|h_0\|_2,
    \]
    and Poincare's inequality makes it bounded in $H^{1}_{0}(I)$. Consequently, there exists a function $u_{\infty} \in H^{1}_{0}(I)$ such that up to a subsequence
    \[
        v_{j}\rightharpoonup u_\infty\quad\text{weakly in }H_0^1(I).
    \]

    \bigskip
    Next we prove that $u_{\infty}$ is in the non-negative cone $\mathcal{K}_{I}$. Since $I=(a,b)$ is bounded, $H^{1}_{0}(I)$ is compactly embedded in $L^{2}(I)$, hence there exists a function $v_{\infty}$ such that up to a subsequence, $v_{j} \rightarrow v_{\infty}$ strongly in $L^{2}(I)$. Note that $v_{j}$ also converges to $u_{\infty}$ weakly in $L^{2}(I)$, it must be that $v_{j} \rightarrow u_{\infty}$ strongly in $L^{2}(I)$. 

    To bound the negative part of $u_{\infty}$, note that $v_{i}\geq 0$ a.e. in $I$, hence
    \begin{equation*}
        0 \leq \max{(0-u_{\infty},0)} \leq v_{j} - u_{\infty},\ a.e.\ in\ I.
    \end{equation*}
    Therefore 
    \begin{equation*}
        \lVert \max{(-u_{\infty},0)} \rVert_{L^{2}(I)} \leq \lVert v_{j} - u_{\infty}\rVert_{L^{2}(I)} \rightarrow 0.
    \end{equation*}
    It follows that $\max{(-u_{\infty},0)}=0$ a.e. in $I$, or equivalently $u_{\infty}$ belongs to the non-negative cone $\mathcal{K}_{I} = \{v\in H_0^1(I):v\geq0\}$.

    \bigskip
    For existence of minimizer, it remains to show that $\mathcal J(u_\infty)=\inf_{v\in\mathcal K_I}\mathcal J(v)$. By definition
    \begin{equation*}
        \begin{split}
            \mathcal{J}(v_{j}) &= \frac12\|h_0+v'_{j}\|_{L^2(I)}^2 = \frac12\|(h_0+u'_{\infty}) + (v'_{j} -u'_{\infty})\|_{L^2(I)}^2 \\
            &= \frac12\|(h_0+u'_{\infty})\|_{L^2(I)}^2 + \int_{I} (h_0+u'_{\infty})(v'_{j} -u'_{\infty})dx + \frac12\| (v'_{j} -u'_{\infty})\|_{L^2(I)}^2\\
            &=\mathcal{J}(u_{\infty}) + \int_{I} (h_0+u'_{\infty})(v'_{j} -u'_{\infty})dx + \frac12\| (v'_{j} -u'_{\infty})\|_{L^2(I)}^2.
        \end{split}
    \end{equation*}
    Therefore
    \begin{equation*}
        \mathcal{J}(u_{\infty}) \leq \mathcal{J}(v_{j}) - \int_{I} (h_0+u'_{\infty})(v'_{j} -u'_{\infty})dx,
    \end{equation*}
    and by definition of weak convergence we have
    \begin{equation}
        \mathcal{J}(u_{\infty}) \leq \liminf_{j\rightarrow\infty}\mathcal{J}(v_{j}) - \liminf_{j\rightarrow\infty}\int_{I} (h_0+u'_{\infty})(v'_{j} -u'_{\infty})dx = \inf_{v\in\mathcal K_I}\mathcal J(v).
    \end{equation}
    Hence $\mathcal{J}(u_{\infty}) = \inf_{v\in\mathcal K_I}\mathcal J(v)$.

    For uniqueness of the minimizer, consider $v,w\in\mathcal K_I$,
    \[
     \mathcal J\!\left(\frac{v+w}{2}\right)
     =\frac{\mathcal J(v)+\mathcal J(w)}2-\frac18\|v'-w'\|_2^2.
    \]
    Two minimizers would therefore have the same derivative. Their difference is constant and has zero trace, so they are equal.

    \bigskip
    To prove the inequalities \eqref{lt:VI} and \eqref{lt:gap}, consider the following weighted average
    \begin{equation*}
        v_{\theta} = (1-\theta)u_{\infty} + \theta v,\qquad \theta\in(0,1).
    \end{equation*}
    Take a finite difference in $\theta$,
    \begin{equation*}
        \frac{J(v_\theta)-\mathcal J(u_\infty)}{\theta}
 =\int_I f_\infty(v'-u_\infty')\,dx
  +\frac{\theta}{2}\|v'-u_\infty'\|_2^2 \geq 0.
    \end{equation*}
    Let $\theta \rightarrow 0$ to obtain inequality \eqref{lt:VI}, and then $\theta = 1$ renders inequality \eqref{lt:gap}.
\end{proof}

In the physics literatures \cite{thorne2017modern}, the physically admissible particle equilibrium $f_{\infty} = h_{0} + u'_{\infty}$ should be radially non-increasing by intuition. That can also be rigorously verified now.
\begin{prop}[Structure of the selected particle equilibrium]\label{lt:plateau}
If $u_{\infty} = \inf_{v\in\mathcal K_I}\mathcal J(v)$ and $f_{\infty} = h_{0} + u'_{\infty}$, then
\begin{enumerate}
    \item The primitive
$F_\infty(x):=\int_a^x f_\infty(s)\,ds$ is concave on $[a,b]$.
    \item In weak form the complementarity relations are
            \[
             u_\infty\geq0,\qquad f_\infty'\leq0,\qquad
             \int_I f'_\infty u_\infty\,dx=0.
            \]
    \item The function $f_\infty$ is constant almost everywhere on each connected component of $\{u_\infty>0\}$.
\end{enumerate}

\end{prop}
\begin{proof}
For $0\leq\phi\in C_c^\infty(I)$ use $v=u_\infty+\phi$ in \eqref{lt:VI}. This gives $\int_I f'_\infty\phi\,dx = -\int_I f_\infty\phi'\,dx\leq0$, or $f_\infty'\leq0$ distributionally. 

Choose a nonnegative smooth mollifier $\rho \in C^{\infty}_{c}((-1,1))$ with $\int_{\mathbb{R}}\rho = 1$, define $\rho_{\epsilon}(z) = \frac{1}{\epsilon}\rho(\frac{z}{\epsilon})$ and set
\[
    F_{\infty,\epsilon}(x):=\int_I\rho_\epsilon(x-y)F_\infty(y)\,dy.
\]
Setting $v(y) = u_{\infty}(y) + \rho_{\epsilon}(x-y)$, by inequality \eqref{lt:VI}, the following inequality holds for any $x \in (a+\epsilon, b - \epsilon)$:
\begin{equation*}
        \begin{split}
            F''_{\infty,\epsilon}(x) &= \int_I\frac{\partial^{2}}{\partial x^{2}}\rho_\epsilon(x-y)F_\infty(y)\,dy\\
            &= \int_I\left(\frac{\partial^{2}}{\partial y^{2}}\rho_\epsilon(x-y)\right)F_\infty(y)\,dy \\
            &= \int_I\rho_\epsilon(x-y)f'_\infty(y)\,dy \\
            &\leq 0,
        \end{split}
\end{equation*}
so $F_{\infty,\epsilon}(x)$ is concave on $(a+\epsilon, b - \epsilon)$. Since $F_\infty$ is continuous,
$F_\varepsilon\to F_\infty$ uniformly on any compact subinterval of $I$.

To see explicitly why concavity passes to the limit, fix $x,y\in I$
and $\theta\in[0,1]$, and put $z=(1-\theta)x+\theta y$. For all
sufficiently small $\varepsilon$, the points $x,y,z$ belong to
$(a+\varepsilon,b-\varepsilon)$, where $F_\varepsilon$ is concave. Hence
\[
 F_\varepsilon((1-\theta)x+\theta y)
 \geq(1-\theta)F_\varepsilon(x)+\theta F_\varepsilon(y).
\]
Letting $\varepsilon\downarrow0$ and using convergence at these three
points gives
\[
 F_\infty((1-\theta)x+\theta y)
 \geq(1-\theta)F_\infty(x)+\theta F_\infty(y).
\]
This is precisely the definition of concavity on $I$. To include the
endpoints, take any $x,y\in[a,b]$ and choose $x_j,y_j\in I$ with
$x_j\to x$ and $y_j\to y$. Apply the last inequality to $x_j,y_j$ and
let $j\to\infty$. Continuity of $F_\infty$ on $[a,b]$ gives the same
inequality for $x,y$. Thus $F_\infty$ is concave on $[a,b]$.

Both $v=0$ and $v=2u_\infty$ belong to $\mathcal K_I$. Applying the
variational inequality \eqref{lt:VI} to these two choices gives
\[
 -\int_I f_\infty u_\infty'\,dx\geq0,
 \qquad
 \int_I f_\infty u_\infty'\,dx\geq0.
\]
The integral must therefore be zero, thus
\begin{equation*}
    \int_I f'_\infty u_\infty\,dx= \int_I f_\infty u'_\infty\,dx = 0.
\end{equation*}

To prove the third property, let $O \coloneqq \{x\in I:u_\infty(x)>0\}$ and consider a nonzero test function $\phi\in C_c^\infty(\{x\in I:u_\infty(x)>0\})$. Its support is a compact set of $O$, hence
\begin{equation*}
    m_\phi:=\min_{\operatorname{supp}\phi}u_\infty>0.
\end{equation*}
Choose $\tau>0$ so small that
$\tau\|\phi\|_\infty\leq m_\phi/2$. Then on
$\operatorname{supp}\phi$ we have
\[
 u_\infty\pm\tau\phi
 \geq u_\infty-\tau|\phi|
 \geq m_\phi/2>0.
\]
Outside that support the functions equal $u_\infty$, and both retain
zero boundary trace. Consequently both
$v_+=u_\infty+\tau\phi$ and $v_-=u_\infty-\tau\phi$ belong to
$\mathcal K_I$. The variational inequality \eqref{lt:VI} yields
\[
 \tau\int_I f_\infty\phi'\,dx\geq0,
 \qquad
 -\tau\int_I f_\infty\phi'\,dx\geq0.
\]
Hence
\[
 \int_I f_\infty\phi'\,dx=0
 \qquad\forall \phi\in C_c^\infty(O),
\]
i.e. $f_\infty'=0$ on $O$ in the sense of distributions.

Finally, let $J=(\alpha,\beta)$ be a connected component of $O$. By definition, $f_\infty|_J\in H^1(J)$ with weak derivative zero. Its absolutely continuous representative $g_J$ therefore satisfies
\[
 g_J(y)-g_J(x)=\int_x^y g_J'(s)\,ds=0,
 \qquad\forall x,y\in J.
\]
Thus $f_\infty$ equals a constant $c_J$ almost everywhere on $J$.
Its primitive $F_\infty$ is therefore affine on $J$.
\end{proof}

\subsection{Construction of the explicit equilibrium}
We are now ready to turn the observation in \ref{subsec:equi_nonunique} into a proposition.
\begin{prop}[Construction of the explicit equilibrium]\label{lt:lcm-equilibrium}
Let $u_\infty\in\mathcal K_I$ be the unique minimizer from
Proposition~\ref{lt:minimizer}, and let $f_\infty=h_0+u_\infty'$. Define
\[
 H(x)=\int_a^x h_0(s)\,ds,\qquad F_\infty(x)=H(x)+u_\infty(x).
\]
Then $F_\infty$ is the least concave majorant (LCM) of $H$ on $[a,b]$, namely
the pointwise smallest concave function that is greater than or equal to
$H$. In particular,
\begin{equation}\label{lt:LCM}
 ,\qquad
 u_\infty=\operatorname{LCM}(H)-H,\qquad
 f_\infty=(\operatorname{LCM}(H))',
\end{equation}
where the derivative identity holds almost everywhere on $I$.

Moreover, on each connected component $(\alpha,\beta)$ of
$\{u_\infty>0\}$, the following equal-area relation holds:
\begin{equation}\label{lt:equal-area}
 f_\infty(x)=c_{\alpha\beta}
   :=\frac{1}{\beta-\alpha}\int_\alpha^\beta h_0(s)\,ds
   \quad\text{for a.e. }x\in(\alpha,\beta).
\end{equation}
\end{prop}
\begin{proof}
Since $u_\infty\geq0$ and $u_\infty(a)=u_\infty(b)=0$, we have
\[
 F_\infty\geq H\quad\text{on }[a,b],\qquad
 F_\infty(a)=H(a),\qquad F_\infty(b)=H(b).
\]
Since $u_\infty(a)=0$ and $u_\infty$ is absolutely continuous,
\[
 F_\infty(x)=H(x)+u_\infty(x)=\int_a^x f_\infty(s)\,ds,
 \qquad F_\infty'=f_\infty\quad\text{a.e. on }I.
\]
By Proposition~\ref{lt:plateau}, $F_{\infty}$ is concave. Since $F_\infty\geq H$, it follows that $F_\infty$ is a concave
majorant of $H$.

It remains to prove that it is the least concave majorant. Let $G$ be any concave majorant of $H$ on $[a,b]$. At every contact point,
where $F_\infty=H$, we immediately have $G\geq F_\infty$. Any other point
lies in a connected component $(\alpha,\beta)$ of the open set
$\{F_\infty>H\}=\{u_\infty>0\}$. The endpoints are contact points:
\[
 F_\infty(\alpha)=H(\alpha),\qquad
 F_\infty(\beta)=H(\beta).
\]

Proposition~\ref{lt:plateau} also says that $f_\infty$ is constant almost
everywhere on $(\alpha,\beta)$, hence $F_\infty$ is affine there. For $x\in(\alpha,\beta)$,
concavity of $G$ therefore gives
\[
 \begin{aligned}
 G(x)
 &\geq\frac{\beta-x}{\beta-\alpha}G(\alpha)
      +\frac{x-\alpha}{\beta-\alpha}G(\beta)\\
      &\geq\frac{\beta-x}{\beta-\alpha}H(\alpha)
      +\frac{x-\alpha}{\beta-\alpha}H(\beta)=\frac{\beta-x}{\beta-\alpha}F_\infty(\alpha)
      +\frac{x-\alpha}{\beta-\alpha}F_\infty(\beta)
  =F_\infty(x).
 \end{aligned}
\]
Thus every concave majorant $G$ lies above $F_\infty$ on all of $[a,b]$.
This proves $F_\infty=\operatorname{LCM}(H)$, and the equal-area relation \eqref{lt:equal-area} follows from the definitions of $F_\infty$ and $f_\infty$.
\end{proof}

\subsection{A bump-on-tail example}
As an illustration of the least-concave-majorant construction, take
$\lambda=1$ and the seeded interval $I=(3/2,6)$. Consider the
normalized bump-on-tail particle distribution
\[
 f_0(x)=\frac{0.9}{\sqrt{2\pi}}e^{-x^2/2}
       +\frac{0.1}{0.4\sqrt{2\pi}}
        \exp\!\left(-\frac{(x-4)^2}{2(0.4)^2}\right).
\]
The first term describes the background population, while the
second describes a narrower bump containing ten percent of the
particles.

To prescribe a small smooth wave seed, define
\[
 \chi(s)=
 \begin{cases}
  \exp\!\left(1-\dfrac{1}{1-s^2}\right),&|s|<1,\\
  0,&|s|\geq1,
 \end{cases}
 \qquad
 W_0(k)=10^{-6}\chi\!\left(\frac{k-k_c}{d}\right),
\]
where
\[
 k_-=\frac{1}{\sqrt{35}},\qquad
 k_+=\frac{2}{\sqrt{5}},\qquad
 k_c=\frac{k_-+k_+}{2},\qquad d=\frac{k_+-k_-}{2}.
\]
Then $W_0\in C_c^\infty(\mathbb R)$, its support is $[k_-,k_+]$,
and $\max W_0=10^{-6}$. Under the resonance map
$k(x)=(x^2-1)^{-1/2}$, the corresponding auxiliary datum is
$\varphi_0(x)=k(x)^3W_0(k(x))$, with positivity set exactly $I$.
Since $\varphi_0(3/2)=0$, the primitive used in the construction is
\[
 H(x)=\int_{3/2}^{x}f_0(s)\,ds-\varphi_0(x).
\]
For these data,
\[
 \left\|H-\int_{3/2}^{\,\cdot}f_0(s)\,ds\right\|_{L^\infty(I)}
 =\|\varphi_0\|_{L^\infty(I)}
 \simeq 2.61\times10^{-7}.
\]
Thus $H$ is almost indistinguishable from the primitive of $f_0$,
while the computation retains the wave correction.

Figure~\ref{fig:bump_on_tail_equilibrium} shows the resulting
majorant and the selected particle density on $I$. It coincides with the physics intuition, as illustrated in the book of Thorne and Blandford \cite{thorne2017modern}. There is one
plateau, with numerically determined contacts and height
\[
 \alpha\simeq2.049918,\qquad
 \beta\simeq4.512386,\qquad
 c\simeq0.0439212.
\]
On $(\alpha,\beta)$, $f_\infty=c$; on the remainder of $I$,
$f_\infty=h_0=f_0-\varphi_0'$. As in the theorem, set
$f_\infty=f_0$ outside $I$.

\begin{figure}[htbp]
 \centering
 \includegraphics[width=\linewidth]{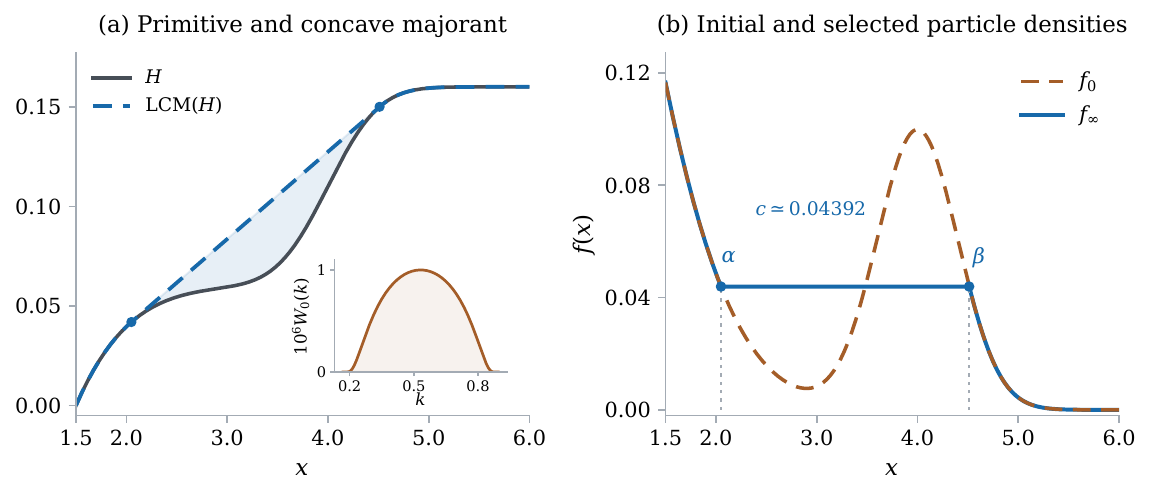}
 \caption{Least-concave-majorant selection for a bump-on-tail
 distribution with a smooth wave seed of peak amplitude $10^{-6}$.
 (a) The corrected primitive $H$ and its least concave majorant;
 their shaded difference is $u_\infty$. The inset shows the
 initial wave spectrum. (b) The initial particle density $f_0$
 and the selected equilibrium
 $f_\infty=(\operatorname{LCM}H)'$ on the seeded interval.
 The affine segment of the majorant produces the plateau between
 $\alpha$ and $\beta$.}
 \label{fig:bump_on_tail_equilibrium}
\end{figure}

\section{Long-time convergence to equilibrium}\label{sec_longtime}

The equilibrium $u_\infty$ has already been constructed in Section~\ref{sec_equilibrium} and characterized by the least-concave-majorant formula \eqref{lt:LCM}. We now prove convergence toward that state.

\subsection{Preparations}
To do estimates with $u(t)$, we need the regularity of $u$ at any fixed time $t$. Similar results have already been proved in \cite{porteous2025existence}, hence we only provide a brief recapitulation here.

\begin{lemma}[Fixed-time regularity and compactness]
\label{lt:trajectory}
For the Part~I approximations posed on $I=(a,b)$, a further
subsequence of the sequence used in the existence proof satisfies
\begin{equation}\label{lt:C-L2}
 u_n\longrightarrow u
 \quad\text{in }C([\tau,T];L^2(I))
 \qquad\text{for every }0<\tau<T<\infty.
\end{equation}
This representative satisfies $u(t)\in\mathcal K_I$ for every
$t>0$. Moreover,
\begin{equation}\label{lt:time-lsc}
 \mathcal J(u(t))
 \leq\liminf_{n\to\infty}
       \frac12\|f_n(t)\|_{L^2(I)}^2
 \qquad\text{for every fixed }t>0.
\end{equation}
\end{lemma}

\begin{proof}
Fix $0<\tau<T<\infty$. Proposition~\ref{prop_grad_bound}
and Poincare's inequality give a uniform spatial $H_0^1(I)$
bound. The maximum principle also gives
$A_n=\gamma(u_n+\varepsilon_n)\leq A_T$ for a constant $A_T$
independent of $n$. Using \eqref{lt:approx-flux} and
\eqref{lt:energy-integrated}, we obtain
\[
 \int_\tau^T\|\partial_tu_n\|_{L^2(I)}^2\,dt
 \leq 2A_T\int_\tau^T D_n(t)\,dt
      +2\varepsilon_n^2(T-\tau)\|g_0\|_{L^2(I)}^2
 \leq C_{\tau,T}.
\]
Consequently,
\begin{equation}\label{lt:time-bounds}
 \sup_{\tau\leq t\leq T}\|u_n(t)\|_{H_0^1(I)}
 +\|\partial_tu_n\|_{L^2((\tau,T)\times I)}
 \leq B_{\tau,T},
\end{equation}
where $B_{\tau,T}$ is independent of $n$.

These are the bounds used in the parabolic interpolation
argument of \cite{porteous2025existence}. They give uniform H\"older moduli of
exponent $1/2$ in space and $1/4$ in time on
$[\tau,T]\times\overline I$. Arzel\`a--Ascoli therefore gives
a subsequence converging uniformly on this cylinder, and hence
in $C([\tau,T];L^2(I))$. The space-time convergence established
in Part~I identifies the limit with $u$. A diagonal extraction
over $[1/m,m]$, $m=2,3,\ldots$, gives \eqref{lt:C-L2}
along a single subsequence and preserves the gradient
convergence used in the existence proof.

For each fixed $t>0$, the spatial bound and the strong $L^2$
convergence imply
\[
 u_n(t)\rightharpoonup u(t)
 \quad\text{in }H_0^1(I).
\]
Since $u_n(t)\geq0$, it follows that $u(t)\in\mathcal K_I$.
Furthermore,
\[
 f_n(t)\rightharpoonup h_0+\partial_xu(t)
 \quad\text{in }L^2(I),
\]
so weak lower semicontinuity proves \eqref{lt:time-lsc}.
\end{proof}

Analogous to the upper barrier proof in \cite{porteous2025existence}, we give a lower barrier for the approximate solutions $u_{n}$. 

\begin{lemma}[First-eigenfunction lower barrier]\label{lt:barrier}
For every compact interval $J\subset\subset I$ and every $T<\infty$, there exists $c_{J,T}>0$, independent of $n$, such that
\[
 u_n(x,t)\geq c_{J,T}\qquad \forall (x,t)\in J\times[0,T].
\]
\end{lemma}

\begin{proof}
Choose $J'=(\alpha,\beta)$ with $J\subset\subset J'\subset\subset I$. Set
\[
 \psi(x)=\sin\frac{\pi(x-\alpha)}{\beta-\alpha},\qquad
 \mu_1=\left(\frac\pi{\beta-\alpha}\right)^2.
\]
Then $\psi''=-\mu_1\psi$, $\psi=0$ on $\partial J'$, and $0<\psi\leq1$ inside. By \eqref{lt:seeded}, choose $0<\kappa\leq\min_{\overline{J'}}\varphi_0$. Define $z(x,t)=\kappa e^{-Ct}\psi(x)$, where
\[
 C=\|\gamma\|_{L^\infty(J')}\mu_1(\kappa+1)+\|(g_0)_-\|_{L^\infty(J')}.
\]
Since $0\leq z\leq\kappa$ and $\varepsilon_n\leq1$, we have
\[
 z_t-\gamma(z+\varepsilon_n)z_{xx}-g_0z
 =z[-C+\mu_1\gamma(z+\varepsilon_n)-g_0]\leq0.
\]
Also $z(0)\leq\varphi_0$ and $z=0\leq u_n$ on the artificial lateral boundary.  Consider the difference $w = z - u_{n}$, then by definition it satisfies
\begin{equation*}
    w_t-\gamma(u_n+\varepsilon_n)w_{xx}-(g_0+\gamma z_{xx})w\leq0.
\end{equation*}
Moreover, $w(\cdot, 0)\leq 0$ and $w\leq 0$ on the boundary $\partial J'$. The reaction rate $g_{0}+ \gamma z_{xx}$ may have either sign, therefore we choose a large enough $M$ such that $\|g_{0} + \gamma z_{xx}\|_{L^{\infty}(J')} + 1< M$, and consider $\widetilde{w} = w e^{Mt}$ which satisfies
\begin{equation*}
    \widetilde{w}_t-\gamma(u_n+\varepsilon_n)\widetilde{w}_{xx}+(M-(g_0+\gamma z_{xx}))\widetilde{w}\leq0, \quad M-(g_0+\gamma z_{xx}) > 1.
\end{equation*}
By definition we have $\widetilde{w}(\cdot, 0)\leq 0$ and $\widetilde{w}\leq 0$ on the boundary $\partial J'$, therefore $\widetilde{w} \leq 0$.
Define
\begin{equation*}
    c_{J,T} = \kappa e^{-Ct}\min_{J}\psi,
\end{equation*}
then $c_{J,T} > 0$ since $J\subset\subset J'$, and by definition
\begin{equation*}
    u_{n} \geq z = \kappa e^{-Ct}\psi(x) \geq \kappa e^{-Ct}\min_{J}\psi =c_{J,T}, \quad \forall (x,t) \in J\times[0,T].
\end{equation*}

\end{proof}

\subsection{Relative entropy}
Given any two nonnegative functions $w(x)$ and $v(x)$, we define the relative entropy of $w$ with respect to $v$ as
\begin{equation}\label{eq:KL_div}
    \mathcal H[w\mid v]
 :=\int_I\frac{1}{\gamma}\left(w-v-v\log(\frac{w}{v})\right)\,dx. =\int_I\frac{1}{\gamma}w\left(1-\frac{v}{w}+\frac{v}{w}\log(\frac{v}{w})\right)\,dx. 
\end{equation}
The above functional is also known as the generalized Kullback-Leibler (KL) divergence, which is commonly used in kinetic equations without mass conservation. Note that since $1-r+r\log{r} \geq 0$ for any $r$, the generalized KL divergence is always nonnegative. 
Let $u(\cdot,t)$ be the solution of the equation 
\begin{equation*}
    u_{t} = \gamma u \partial_{x}(\partial_{x}u+h_{0}),
\end{equation*}
then formally speaking, we have the following relative entropy identity
\begin{equation*}
    \frac{d}{dt}\mathcal H[u(t)\mid v]
 +\frac12\|f_{u(t)}\|_{L^2(I)}^2-\frac12\|f_v\|_{L^2(I)}^2
 +\frac12\|f_{u(t)}-f_v\|_{L^2(I)}^2=0,
\end{equation*}
where $f_{v} = v' + h_{0}$. Or by definition of the variational energy $\mathcal{J}$, the above identity can be written as
\begin{equation}\label{eq:rel_entropy}
    \frac{d}{dt}\mathcal H[u(t)\mid v]
 +\mathcal J(u(t))-\mathcal J(v)
 +\frac12\|f_{u(t)}-f_v\|_{L^2(I)}^2=0.
\end{equation}
However, for the weak solution constructed in Part~I, 
\eqref{eq:rel_entropy} cannot be used directly. The solution may vanish, the relative entropy may be singular, and the required differentiation is not justified. Therefore in what follows we carry out the argument at the approximate level using a regularized relative entropy.

Let $u_{n}$ be the solution of \eqref{lt:approx-flux} and $f_{n}$ be $u'_{n} + h_{0}$, define the regularized relative entropy as
\begin{equation}\label{Cal Hn}
    \mathcal{H}_{n}\left[u_{n}(t) \mid v\right] \coloneqq  \int_{I}\frac{1}{\gamma}\left( u_{n}(t) - v - (v+\varepsilon_{n})\log{\frac{u_{n}(t)+\varepsilon_{n}}{v+\varepsilon_{n}}} \right)dx, 
 \end{equation}   
{then} 
\begin{equation*}
    \frac{d}{dt}\left(\mathcal{H}_{n}\left[u_{n}(t) \mid v\right]\right) = \int_{I} (u_n-v)\partial_x f_n\,dx -\varepsilon_n\int_I
 \frac{(u_n-v)g_0}
      {\gamma(u_n+\varepsilon_n)}\,dx. \nonumber
\end{equation*}
Analogous to \eqref{eq:KL_div}, it can be easily verified that the regularized relative entropy is also nonnegative.

Denote the residual as
\begin{equation*}
    R_{n,v}(t)
 =
 -\varepsilon_n\int_I
 \frac{(u_n-v)g_0}
      {\gamma(u_n+\varepsilon_n)}\,dx,
\end{equation*}
then we obtain
\begin{equation}\label{eq:regu_entropy_id}
    \frac{d}{dt}\left(\mathcal{H}_{n}\left[u_{n}(t) \mid v\right]\right) +  \frac12\|f_n(t)\|_{L^{2}(I)}^2-  \frac12\|f_v\|_{L^{2}(I)}^2+ \frac12\|f_n(t)-f_v\|_{L^{2}(I)}^2 = R_{n,v}(t).
\end{equation}
This is the rigorous regularized counterpart of
\eqref{eq:rel_entropy}; the only additional term is the residual
$R_{n,v}$, which vanishes as $n\to\infty$.

\begin{prop}[Relative entropy estimate]\label{lt:combined-estimate}
For every $t>0$ and every $0\leq v\in C_c^\infty(I)$, the limiting trajectory $u(t)$ with initial condition $u(\cdot,0) = \varphi_{0}(\cdot)$ satisfies
\begin{equation}\label{lt:master}
 \mathcal J(u(t))
 \leq
 \mathcal J(v)
 +\frac{\mathcal H[\varphi_0\mid v]}{t}.
\end{equation}
\end{prop}

\begin{proof}
Denote the approximate regularized relative entropy $\mathcal{H}_{n}\left[u_{n}(t) \mid v\right]$ as $\mathcal{H}_{n,v}(t)$. Integrating the regularized relative entropy identity \eqref{eq:regu_entropy_id}
 gives
 \begin{align*}
 \mathcal{H}_{n,v}(t)-\mathcal{H}_{n,v}(0)
 &+\int_0^t\bigl(\frac12\|f_n(s)\|_{L^{2}(I)}^2-\mathcal J(v)\bigr)\,ds\\
 &+\frac12\int_0^t
   \|f_n(s)-f_v\|_2^2\,ds
 =\int_0^tR_{n,v}(s)\,ds.
\end{align*}
 Since both $\mathcal{H}_{n,v}(t)$ and $\frac12\int_0^t
   \|f_n(s)-f_v\|_2^2\,ds$ are nonnegative, we have
   \begin{equation*}
       \int_0^t\frac12\|f_n(s)\|_{L^{2}(I)}^2\,ds \leq t\mathcal J(v) + \mathcal{H}_{n,v}(0) + \int_0^tR_{n,v}(s)\,ds.
   \end{equation*}
 Note that  \eqref{lt:energy-integrated} implies, for
$0\leq s\leq t$,
\[
 \frac12\|f_n(t)\|_{L^{2}(I)}^2
 \leq
 \frac12\|f_n(s)\|_{L^{2}(I)}^2+C_0\varepsilon_n(t-s).
\]
Integrating with respect to $s\in(0,t)$ yields
\[
  \frac{t}{2}\|f_n(t)\|_{L^{2}(I)}^2
 \leq
 \int_0^t\frac12\|f_n(s)\|_{L^{2}(I)}^2\,ds
 +\frac{C_0\varepsilon_n t^2}{2},
\]
therefore
\begin{equation}\label{eq:regu_entropy_integral}
    \frac{t}{2}\|f_n(t)\|_{L^{2}(I)}^2
 \leq t\mathcal J(v) + \mathcal{H}_{n,v}(0) + \int_0^tR_{n,v}(s)\,ds + \frac{C_0\varepsilon_n t^2}{2}.
\end{equation}
By Lemma \ref{lt:trajectory}, $t\mathcal J(u(t))
\leq\liminf_{n\to\infty}\frac{t}{2}\|f_n(t)\|_{L^2(I)}^2$, and the fourth term on the right hand side converges to $0$ as $n$ goes to infinity. Hence it remains to show that
 \begin{equation*}
     \lim_{n\rightarrow \infty} \mathcal{H}_{n,v}(0) = \mathcal{H}\left[\varphi_{0}\mid v\right],
 \end{equation*}
 and
 \begin{equation*}
     \lim_{n\rightarrow \infty} \int_0^tR_{n,v}(s)\,ds = 0.
 \end{equation*}

    We prove the second limit first. Fix $0\leq v\in C_c^\infty(I)$ and $T<\infty$, and choose
$J\subset\subset I$ containing $\operatorname{supp}v$. Let
\[
 C_g=\left\|\frac{g_0}{\gamma}\right\|_{L^\infty(I)}.
\]
By Lemma~\ref{lt:barrier},
\[
 |R_{n,v}(t)|
 \leq
 C_g\varepsilon_n
 \left(
 |I|+\frac{\|v\|_{L^1(I)}}{c_{J,T}}
 \right),
 \qquad 0\leq t\leq T.
\]
Indeed,
\[
 \frac{u_n}{u_n+\varepsilon_n}\leq1
\]
everywhere, while the positive lower bound for $u_n$ is needed only on
$\operatorname{supp}v$.  Consequently we have
\begin{equation*}
  \lvert\int_0^tR_{n,v}(s)\,ds  \rvert \leq  \int_0^t|R_{n,v}(s)|\,ds \leq t C_g\varepsilon_n
 \left(
 |I|+\frac{\|v\|_{L^1(I)}}{c_{J,t}}
 \right)
 \longrightarrow0.
\end{equation*}
And it follows that $\lim_{n\rightarrow \infty} \int_0^tR_{n,v}(s)\,ds = 0$.

Next we prove the convergence of initial relative entropy. Recall the definition,
\begin{equation*}
    \mathcal{H}\left[\varphi_{0}\mid v\right] = \int_I\frac{1}{\gamma}\left(\varphi_{0}-v-v\log{\varphi_{0}}+v\log{v}\right)\,dx.
\end{equation*}
Since $\varphi_{0}$ has a positive minimum on $\text{supp }v$,  and $v\log{v}$ is continuous at zero, we have
\begin{equation*}
    \mathcal{H}\left[\varphi_{0}\mid v\right] < \infty.
\end{equation*}
Denote the approximate regularized relative entropy $\mathcal{H}_{n}\left[u_{n}(t) \mid v\right]$ as $\mathcal{H}_{n,v}(t)$, then
\begin{equation*}
    \mathcal{H}_{n,v}(0) = \int_{I}\frac{1}{\gamma}\left( \varphi_{0} - v - (v+\varepsilon_{n})\log{\frac{\varphi_{0}+\varepsilon_{n}}{v+\varepsilon_{n}}} \right)dx.
\end{equation*}
Taking the partial derivative with respect to $\varepsilon_{n}$, we have
\begin{equation*}
    \frac{\partial \mathcal{H}_{n,v}(0)}{\partial \varepsilon_{n}} =\int_{I} \frac{1}{\gamma}\left(1 - \log{\left(\frac{\varphi_{0}+\varepsilon_{n}}{v+\varepsilon_{n}} \right)} - \frac{v+\varepsilon_{n}}{\varphi_{0}+\varepsilon_{n}} \right),
\end{equation*}
Since $1 - \log{r} - \frac{1}{r} \leq 0$ for any $r>0$, we have $\frac{\partial \mathcal{H}_{n,v}(0)}{\partial \varepsilon_{n}} \leq 0$, and
\begin{equation*}
    \lim_{n\rightarrow \infty} \mathcal{H}_{n,v}(0) = \lim_{\varepsilon_{n}\rightarrow 0}\mathcal{H}_{n,v}(0) = \mathcal{H}\left[\varphi_{0}\mid v\right]
\end{equation*}
follows from the monotone convergence theorem.
\end{proof}

\subsection{Proof of Theorem \ref{lt:mainB}}
If we replace $v$ with $u_{\infty}$ straightforwardly in \eqref{lt:master}, the relative entropy $\mathcal{H}[\varphi_{0}\mid u_{\infty}]$ can be infinite, and the inequality becomes trivial. Therefore to prove long-time convergence to equilibrium, we first approximate $u_{\infty}$ by nonnegative smooth functions $v_{m}$ such that every term in \eqref{lt:master} is well-defined.

Set $w_{\delta} = (u_{\infty} - \delta)_{+}$, then $w_{\delta}$ is nonnegative and compactly supported in $I = (a,b)$. Moreover
\begin{equation*}
    w'_{\delta} = \mathbf{1}_{\{u_{\infty}>\delta\}}u'_{\infty}\quad\text{a.e.}
\end{equation*}
and by dominated convergence,
\begin{equation*}
    \lim_{\delta\rightarrow 0} \|w_\delta'-u'_{\infty}\|_{L^{2}(I)}^2
 =\lim_{\delta\rightarrow 0} \int_{\{u_{\infty}\leq\delta\}}|u'_{\infty}|^2\,dx
 = \int_{\{u_{\infty}=0\}}|u'_{\infty}|^2\,dx=0.
\end{equation*}
By Poincare's inequality we have $w_{\delta} \rightarrow u_{\infty}$ in $H_{0}^{1}(I)$.

Set $\delta_{m} = \frac{1}{m}$, and let $\widetilde{w}_{m}$ be the zero extension of $w_{\delta_{m}}$ to $\mathbb{R}$. There must exist a $d_{m} \in (0, \frac{b-a}{2})$ such that 
\begin{equation*}
    \text{supp }\widetilde{w}_{m} \subset [a+d_{m}, b  - d_{m}].
\end{equation*}
Choose the width of mollifier $\sigma_{m}$ small enough, such that
\[
 \sigma_m<\min\{d_m/2,1/m\},\qquad
 \|(\rho_{\sigma_m}*\widetilde w_m)'-\widetilde w_m'\|_{L^2(\mathbb R)}
 <\frac1m,
\]
where $\rho\in C_c^\infty((-1,1))$, $\int_{\mathbb R}\rho=1$, and
$\rho_{\sigma_{m}}(x)=\sigma_{m}^{-1}\rho(x/\sigma_{m})$. 

We denote $\rho_{\sigma_m}*\widetilde w_m$ as $v_{m}$, then by definition 
\begin{equation*}
    \text{supp } v_{m} \subset [a+d_{m} - \sigma_{m}, b - d_{m} + \sigma_{m}] \subset [a+\frac{d_{m}}{2}, b - \frac{d_{m}}{2}] \subset\subset I.
\end{equation*}
Thus $0\leq v_m\in C_c^\infty(I)$. The derivative error is bounded by
the sum of the mollification error and the truncation error:
\begin{equation*}
    \|v_m'-u_{\infty}'\|_{L^2(I)}
 \leq\|v_m'-w_{\delta_m}'\|_{L^2(I)}
       +\|w_{\delta_m}'-u_{\infty}'\|_{L^2(I)}\leq\frac1m+\|w_{\delta_m}'-u_{\infty}'\|_{L^2(I)}
 \longrightarrow0.
\end{equation*}
Poincare's inequality gives $v_{m} \rightarrow u_{\infty}$ in $H^{1}_{0}(I)$.

Now we have constructed the approximation sequence $v_{m}$, substitute into the relative entropy estimate \eqref{lt:master} to obtain
\begin{equation*}
    \mathcal{J}(u(t)) \leq \mathcal{J}(v_{m}) + \frac{\mathcal{H}[\varphi_{0}\mid v_{m}]}{t}.
\end{equation*}
Since $u_{\infty}$ is the minimizer of $\mathcal{J}$, we have
\begin{equation}\label{lt:squeeze}
    0 \leq \mathcal{J}(u(t)) - \mathcal{J}(u_{\infty})\leq \mathcal{J}(v_{m}) - \mathcal{J}(u_{\infty}) + \frac{\mathcal{H}[\varphi_{0}\mid v_{m}]}{t}.
\end{equation}

Fix $m$, since $\mathcal{H}[\varphi_{0}\mid v_{m}]$ is finite and independent of $t$, taking the upper limit as $t\rightarrow \infty$ in \eqref{lt:squeeze} gives
\[
 \limsup_{t\to\infty}
 \bigl[\mathcal J(u(t))-\mathcal J(u_\infty)\bigr]
 \leq \mathcal J(v_m)-\mathcal J(u_\infty).
\]
The right hand side tends to zero as $m$ goes to infinity. Indeed
\begin{equation*}
   \begin{split}
       | \mathcal{J}(v_{m}) - \mathcal{J}(u_{\infty})| = &\left|\int_I(h_0+u'_{\infty})(v_m'-u'_{\infty})\,dx
            +\frac12\|v_m'-u'_{\infty}\|_{L^{2}(I)}^2\right|\\
            \leq & \|h_0+u'_{\infty}\|_{L^{2}(I)}^2\|v_m'-u'_{\infty}\|_{L^{2}(I)}^2 + \frac12\|v_m'-u'_{\infty}\|_{L^{2}(I)}^2 \rightarrow 0.
   \end{split}
\end{equation*}
Since the energy gap is nonnegative, it follows that
\[
 0
 \leq \liminf_{t\to\infty}
       \bigl[\mathcal J(u(t))-\mathcal J(u_\infty)\bigr]
 \leq \limsup_{t\to\infty}
       \bigl[\mathcal J(u(t))-\mathcal J(u_\infty)\bigr]
 \leq 0.
\]
Hence 
\[
 \lim_{t\to\infty}\mathcal J(u(t))
 =\mathcal J(u_\infty).
\]

\bigskip
We now convert energy convergence into strong convergence.
In \eqref{lt:gap} let $v=u(t)$, we obtain
\[
 \frac12\|f(t)-f_\infty\|_{L^2(I)}^2
 =
 \frac12\|u'(t)-u_\infty'\|_{L^2(I)}^2
 \leq
 \mathcal J(u(t))-\mathcal J(u_\infty)
 \longrightarrow 0.
\]
Therefore,
\[
 f(t)\longrightarrow f_\infty
 \qquad\text{strongly in }L^2(I).
\]
and Poincare's inequality gives $u(t)$ converging to $u_{\infty}$ strongly in $H^{1}_{0}(I)$ as $t$ goes to infinity.

\subsection{Rate of convergence under finite-entropy assumption}
Under an extra assumption, we are able to give a quantitative estimate on the convergence rate to equilibrium
\begin{coro}[Convergence rate]
    If $\mathcal H[\varphi_0\mid u_\infty]<\infty$, then $\|f(t)-f_\infty\|_{L^2}=O(t^{-1/2})$.
\end{coro}
\begin{proof}
    Take \eqref{eq:regu_entropy_integral} and replace $v$ with $u_{\infty}$ to obtain
    \begin{equation*}
        \frac{1}{2}\|f_n(t)\|_{L^{2}(I)}^2
 \leq \mathcal J(u_{\infty}) + \frac{1}{t}\left(\mathcal{H}_{n}(0) + \int_0^tR_{n}(s)\,ds + \frac{C_0\varepsilon_n t^2}{2}.\right)
    \end{equation*}
    where
    \begin{equation*}
        \mathcal{H}_{n}(0) = \int_{I}\frac{1}{\gamma}\left( \varphi_{0} - u_{\infty} - (u_{\infty}+\varepsilon_{n})\log{\frac{\varphi_{0}+\varepsilon_{n}}{u_{\infty}+\varepsilon_{n}}} \right)dx.
    \end{equation*}
    and 
    \begin{equation*}
        R_{n}(t)
 =
 -\varepsilon_n\int_I
 \frac{(u_n-u_{\infty})g_0}
      {\gamma(u_n+\varepsilon_n)}\,dx.
    \end{equation*}
    Taking the lower limit gives
    \begin{equation*}
        \mathcal{J}(u(t)) - \mathcal{J}(u_{\infty}) \leq \liminf_{n\rightarrow \infty} \left(\frac{1}{2}\|f_n(t)\|_{L^{2}(I)}^2
 - \mathcal J(u_{\infty})\right) \leq \frac{H_{\infty}}{t}
    \end{equation*}
    By inequality \eqref{lt:gap}, we have
    \begin{equation*}
        \frac{1}{2}\|f(t)-f_\infty\|_{L^{2}(I)}^2 \leq \mathcal{J}(u(t)) - \mathcal{J}(u_{\infty}) \leq \frac{H_{\infty}}{t},
    \end{equation*}
    hence $\|f(t)-f_\infty\|_{L^{2}(I)}^2 = \mathcal{O}(t^{-1/2})$.
\end{proof}

\section{Conclusion}\label{sec:conclusion}
In this work, we have established the existence of global weak solutions
for the one-dimensional electrostatic quasilinear theory with a prescribed
dispersion relation. Uniform estimates for
regularized approximations, together with strong convergence of their
spatial gradients, allowed us to pass to the nonlinear limit and recover weak solutions for the particle distribution and wave spectrum.

We have also characterized the long-time equilibrium selected by this
approximation procedure when the initial wave spectrum is positive
throughout a single connected wave-number band. Although the degeneracy
of the equation permits multiple stationary states, the selected
auxiliary profile is the unique minimizer of a quadratic functional over
a nonnegative cone. The geometric construction provides an explicit description
of the equilibrium for prescribed initial data.

Finally, a regularized relative-entropy argument and a variational
energy-gap inequality yield $L^{2}$-convergence of the particle distribution. Qualitative convergence requires
neither smoothness of the equilibrium nor finite relative entropy with
respect to it; the latter condition, when imposed, gives an
$O(t^{-1/2})$ convergence rate.

Physical admissibility, wellposedness and regularity of the solution will be discussed in future work.

\bigskip
\noindent {\bf Acknowledgements}\\
placeholder

\bigskip
\noindent {\bf Conflict of Interest}\\
The authors have no conflicts to disclose.

\bigskip
\noindent {\bf Data Availability}\\
No new data were created or analyzed in this study, therefore data sharing is not applicable to this article.

\bibliographystyle{plain}
\bibliography{main.bib}

    \section*{Appendix: Existence of classical solutions for problem (\ref{eq_Sn})}

According to Theorem 6.1 in Chapter V of Ladyzenskaja et. al.'s book \cite{ladyzhenskaya1968linear}, the following conditions (a) to (f) are sufficient for Theorem \ref{thm_classical_sol}. 

Recall that
\begin{equation*}
    \begin{split}a_{1}(x,t,u,p) & \coloneqq\gamma\mathcal{P}_{n}(u)p,\\
\tilde{a}(x,t,u,p) & \coloneqq\gamma\mathcal{P}_{n}^{\prime}(u)p^{2}+\left(\partial_{x}\gamma\right)\mathcal{P}_{n}(u)p-g_{0}(x)u,\\
A(x,t,u,p) & \coloneqq-g_{0}(x)u,
\end{split}
\end{equation*}
and

    \begin{equation}\label{psi_def}
\varphi(x,t)\coloneqq\varphi_{0}(x)+\left[\gamma\mathcal{P}_{n}(\varphi_{0})\partial_{x}^{2}\varphi_{0}+g_{0}\varphi_{0}\right]t.
\end{equation}

We will verify the conditions one by one.

\begin{enumerate}[label=(\alph*)]
    \item For $(x,t) \in \overline{Q_{T}}$ and arbitrary $u$, the diffusion term is strictly coercive,
    \begin{equation*}
        \frac{\partial a_{1}}{\partial p}(x,t,u,p)=\gamma\mathcal{P}_{n}(u)\geq\frac{\gamma_{m}}{2n}>0,
    \end{equation*}
    and the reaction term has the following lower bound,
    \begin{equation*}
        A(x,t,u,0)u=-g_{0}(x)u^{2}\geq-\max(|g_{0}|)u^{2}.
    \end{equation*}

    \item  For $(x,t) \in \overline{Q_{T}}$, when $|u|\leq M$, for arbitrary $p$, the operators are bounded in the following sense.

    \begin{equation*}
        \frac{\partial a_{1}}{\partial p}(x,t,u,p) =\gamma\mathcal{P}_{n}(u) \leq \max{\left(\gamma\right)}\left(M+1\right),
    \end{equation*}
    and
    \begin{equation*}
        \begin{split} & \left(\lvert a_{1}\rvert+\lvert\frac{\partial a_{1}}{\partial u}\rvert\right)\left(1+\lvert p\rvert\right)+\lvert\frac{\partial a_{1}}{\partial x}\rvert+\lvert\tilde{a}\rvert\\
= & \left(\gamma\mathcal{P}_{n}(u)\lvert p\rvert+\gamma\mathcal{P}_{n}^{\prime}(u)\lvert p\rvert\right)\left(1+\lvert p\rvert\right)+\left(\partial_{x}\gamma\right)\mathcal{P}_{n}(u)\lvert p\rvert+\lvert\gamma\mathcal{P}_{n}^{\prime}(u)p^{2}+\left(\partial_{x}\gamma\right)\mathcal{P}_{n}(u)p-g_{0}(x)u\rvert\\
\leq & \left(\gamma\mathcal{P}_{n}(u)\lvert p\rvert+\gamma\mathcal{P}_{n}^{\prime}(u)\lvert p\rvert\right)\left(1+\lvert p\rvert\right)+\left(\partial_{x}\gamma\right)\mathcal{P}_{n}(u)\lvert p\rvert+\gamma\mathcal{P}_{n}^{\prime}(u)p^{2}+\left(\partial_{x}\gamma\right)\mathcal{P}_{n}(u)\lvert p\rvert+\lvert g_{0}(x)u\rvert\\
\leq & \mu(M,b,\max(|g_{0}|))\left(1+\lvert p\rvert\right)^{2}.
\end{split}
    \end{equation*}
    \item For $(x,t) \in \overline{Q_{T}}$, $|u| \leq M$ and $|p| \leq M_{1}$, the functions $a_{1}$, $\tilde{a}$, $\frac{\partial a_{1}}{\partial p}$, $\frac{\partial a_{1}}{\partial u}$, and $\frac{\partial a_{1}}{\partial x}$ are arbitrarily smooth in $x$, $t$, $u$ and $p$, therefore they satisfy any H\"older continuity condition. 
    \item 
    Note that
    \begin{equation*}
        \begin{split}\frac{\partial a_{1}}{\partial u} & =\gamma\mathcal{P}_{n}^{\prime}(u)p,\\
\frac{\partial\tilde{a}}{\partial p} & =2\gamma\mathcal{P}_{n}^{\prime}(u)p+\left(\partial_{x}\gamma\right)\mathcal{P}_{n}(u),\\
\frac{\partial\tilde{a}}{\partial u} & =\gamma\mathcal{P}_{n}^{\prime\prime}(u)p^{2}+\left(\partial_{x}\gamma\right)\mathcal{P}_{n}^{\prime}(u)p-g_{0}(x).
\end{split}
    \end{equation*}
    For $(x,t) \in \overline{Q_{T}}$, $|u| \leq M$ and $|p| \leq M_{1}$, all the above terms are bounded by a constant $C(M, M_{1}, \mathcal{P}_{n}, g_{0}, \Omega)$.
    
    In addition, neither $\tilde{a}$ nor $a_{1}$ depend on $t$, therefore condition (d) is satisfied.
    \item By definition of $\varphi $ in Equation(\ref{psi_def}), $\varphi $ is arbitrarily smooth in $\overline{Q_{T}}$. In addition, for $x\in \partial \Omega$ and $t=0$, the following identity holds,

        \begin{equation*}
    \partial_{t}\varphi(x,t)=\gamma\mathcal{P}_{n}(\varphi_{0})\partial_{x}^{2}\varphi_{0}+g_{0}\varphi_{0}=\gamma\mathcal{P}_{n}(\varphi)\partial_{x}^{2}\varphi+g_{0}\varphi.
    \end{equation*}
    
    \item It is trivial that the boundary $\partial \Omega$ satisfies any H\"older continuity condition.
\end{enumerate}

\end{document}